\documentclass[11pt]{amsart}
\usepackage{amsmath}
\usepackage{amsthm}
\usepackage{amssymb}
\usepackage{a4wide}
\usepackage{color}

\usepackage[pagewise,mathlines]{lineno}
\usepackage{hyperref}
\usepackage{amsfonts}

\usepackage{mathtools}
\mathtoolsset{showonlyrefs}

\newcommand*{\Z}{\mathbb Z}

\newcommand*{\R}{\mathbb R}

\newtheorem{lemma}{Lemma}[section]
\newtheorem{cor}{Corollary}[section]

\newtheorem{theorem}{Theorem}[section]

\newtheorem{remark}{Remark}[section]

\begin{document}

\title[Nonlocal problems in perforated domains]
{Nonlocal problems in perforated domains}

\author[M. C. Pereira and
J. D. Rossi]{Marcone C. Pereira and
Julio D. Rossi}

\address{Julio D. Rossi
\hfill\break\indent Dpto. de Matem{\'a}ticas, FCEyN,
\hfill\break\indent
Universidad de Buenos Aires, \hfill\break\indent Ciudad Universitaria Pab 1, 1428, Buenos Aires,
Argentina. } \email{{\tt jrossi@dm.uba.ar} \hfill\break\indent {\it
Web page: }{\tt http://mate.dm.uba.ar/$\sim$jrossi/}}

\address{Marcone C. Pereira
\hfill\break\indent Dpto. de Matem{\'a}tica Aplicada, IME,
\hfill\break\indent
Universidade de S\~ao Paulo, \hfill\break\indent Rua do Mat\~ao 1010, 
S\~ao Paulo - SP, Brazil. } \email{{\tt marcone@ime.usp.br} \hfill\break\indent {\it
Web page: }{\tt www.ime.usp.br/$\sim$marcone}}


\begin{abstract} 
In this paper we analyze nonlocal equations in perforated domains. We consider nonlocal problems
of the form $f(x) = \int_{B} J(x-y) (u(y) - u(x)) dy$ with $x$ in a perforated domain $\Omega^\epsilon \subset \Omega$. Here
$J$ is a non-singular kernel. 
We think about $\Omega^\epsilon$ as a fixed set $\Omega$ from where we have
removed a subset that we call the holes.
We deal both with the Neumann and Dirichlet conditions in the holes and assume a Dirichlet condition outside $\Omega$. In the later case we impose that
$u$ vanishes in the holes but integrate in the whole $\R^N$ ($B=\R^N$) and in the former we just consider integrals in $\R^N$ minus the holes ($B=\R^N \setminus (\Omega \setminus \Omega^\epsilon)$).
Assuming weak convergence of the holes, specifically, under the assumption that the characteristic function of $\Omega^\epsilon$ has a weak limit, $\chi_{\epsilon} \rightharpoonup \mathcal{X}$ weakly$^*$ in $L^\infty(\Omega)$,
we analyze the limit as $\epsilon \to 0$ of the solutions to the nonlocal problems proving that there is a
nonlocal limit problem. In the case in which the holes are periodically removed balls we obtain that
the critical radius is of order of the size of the typical cell (that gives the period). 
In addition, in this periodic case, we also study the behavior of these nonlocal problems when we rescale the kernel 
in order to approximate local PDE problems. 
\end{abstract}

\maketitle

\section{Introduction}
\label{Sect.intro}
\setcounter{equation}{0}

Let $\Omega^\epsilon \subset \R^N$ be a family of open bounded sets satisfying 
$$\Omega^\epsilon \subset \Omega$$
for some fixed open bounded domain $\Omega \subset \R^N$ and $\epsilon>0$. 
If $\chi_\epsilon \in L^\infty(\R^N)$ is the characteristic function of $\Omega^\epsilon$, 
we also assume that there exists $\mathcal{X} \in L^\infty(\R^N)$ such that 
\begin{equation} \label{CharF}
\chi_\epsilon \rightharpoonup \mathcal{X}  \quad  \textrm{ weakly$^*$ in } L^\infty(\Omega).
\end{equation}
This means,
$$
\int_{\Omega} \chi_\epsilon(x) \, \varphi(x) \, dx \to \int_{\Omega} \mathcal{X}(x) \, \varphi(x) \, dx \quad \textrm{ as } \epsilon \to 0
$$
for all $\varphi \in L^1(\Omega)$.
Note that both functions $\chi_\epsilon$ and $\mathcal{X}$ satisfy  
$$
\begin{gathered}
0 \leq \mathcal{Y}(x) \leq 1, \textrm{ for all } x \in \R^N \\
 \mathcal{Y}(x) \equiv 0 \quad \textrm{ as } x \in \R^N \setminus \Omega,
 \end{gathered}
 $$
 for $\mathcal{Y} = \chi_\epsilon$ or $\mathcal{X}$.

Our main goal in this paper is to study nonlocal problems with non-singular kernels in the \emph{perforated domains} $\Omega^\epsilon$. 
We consider problems of the form
 $$f (x) = \int_{B} J(x-y) (u^\epsilon (y) - u^\epsilon (x)) dy$$ 
 with $ x \in \Omega^\epsilon \subset \Omega$. Here
$J$ is a non-singular kernel. 
We deal both with the Neumann and Dirichlet problems. 
For the Dirichlet case we impose that
$u$ vanishes in $\R^N \setminus \Omega^\epsilon$ and we integrate in the whole $\R^N$ ($B=\R^N$) 
while in the Neumann case we just consider integrals in $\R^N$ minus $\Omega \setminus \Omega^\epsilon$ ($B=\R^N \setminus (\Omega \setminus \Omega^\epsilon)$) only assuming that $u$ vanishes in $\R^N\setminus \Omega$.
Note that for this last case we have considered nonlocal Neumann boundary conditions
in the holes and a Dirichlet boundary condition in the exterior of the set $\Omega$.

Along the whole paper, we assume that the function $J$ that appears as the kernel in the nonlocal problem satisfies the following hypotheses:
$$
{\bf (H_J)} \qquad 
\begin{gathered}
J \in \mathcal{C}(\R^N,\R) \textrm{ is non-negative and compactly supported with } 
J(0)>0, \\ \; J(-x) = J(x) \textrm{ for every $x \in \R^N$, and } 
\int_{\R^N} J(x) \, dx = 1.
\end{gathered}
$$
On the other hand we only assume that $f\in L^2 (\Omega)$.

Our main result, that holds both for the Dirichlet and the Neumann problem, says that 
there exists a limit as $\epsilon \to 0$,
$$
\tilde u^\epsilon \rightharpoonup u^*, \quad \textrm{ weakly in } L^2 (\Omega),
$$
where $\, \tilde \cdot \,$ denotes the extension by zero of functions defined in subsets of $\R^N$. 
We characterize the nonlocal problem that verifies the weak limit $u^*$. 
Note that in the Dirichlet problem, the extension by zero and the solutions coincide,
and then, to consider $\, \tilde u \,$ can be omitted.

For the Dirichlet problem we have the following result:
\begin{theorem} \label{theoD}
Let $\{ u^\epsilon \}_{\epsilon>0}$ be the family of solutions of the nonlocal Dirichlet problem
\begin{equation} \label{1.1Dintro}
f(x) = \int_{\R^N } J (x-y) (u^\epsilon (y) - u^\epsilon (x)) dy, \qquad x  \in \Omega^\epsilon
\end{equation}
with 
\begin{equation} \label{BCDintro}
u^\epsilon (x) \equiv 0, \qquad x \in \R^N \setminus \Omega^\epsilon,
\end{equation}
for $f \in L^2(\Omega)$, and assume that the characteristic functions $\chi_\epsilon$ satisfy \eqref{CharF}.
Then, there exists $u^* \in L^2(\Omega)$ such that 
$$\tilde u^\epsilon \rightharpoonup u^* \textrm{ weakly in } L^2(\Omega).$$
Moreover, the limit $u^*$ satisfies the following nonlocal problem in $\Omega$,
$$
\mathcal{X}(x) \, f(x) = \mathcal{X}(x) \int_{\R^N}  J (x-y) (u^*(y) - u^*(x) ) \, dy - (1-\mathcal{X}(x)) \, u^*(x) 
$$
with 
$$
u^*(x) \equiv 0, \qquad x \in \R^N \setminus \Omega.
$$
\end{theorem}

We want to remark that no regularity assumptions on the sets $\Omega^\epsilon$ (besides measurability and the weak convergence \eqref{CharF}) is needed
for our arguments.

For local operators, like the usual Laplacian, i.e., for the problem $\Delta v^\epsilon = f$ in $\Omega^\epsilon$
with $v^\epsilon =0$ on $\partial \Omega^\epsilon$, the study of the behavior of 
solutions in perforated domains has attracted much interest since the pioneering works \cite{CH, Necas, RT}. 
In the classical paper \cite{CM}, for example, the authors
consider the Dirichlet problem for the equation $\Delta v^\epsilon = f$ in a bounded domain 
from where we have removed a big number of periodic small balls (the holes). 
That is, they consider 
$$\Omega^\epsilon = \Omega \setminus \cup_i B_{r^\epsilon} (x_i)$$
where $B_{r^\epsilon} (x_i)$ is a ball centered 
in $x_i\in \Omega$ of the form $x_i \in 2 \epsilon \Z^N$ with radius $0 < r^\epsilon < \epsilon \leq 1$. 
See Figure \ref{fig1} bellow for an example of a periodic perforated domain $\Omega^\epsilon$.

\begin{figure}[htp] 
\centering \scalebox{0.4}{\includegraphics{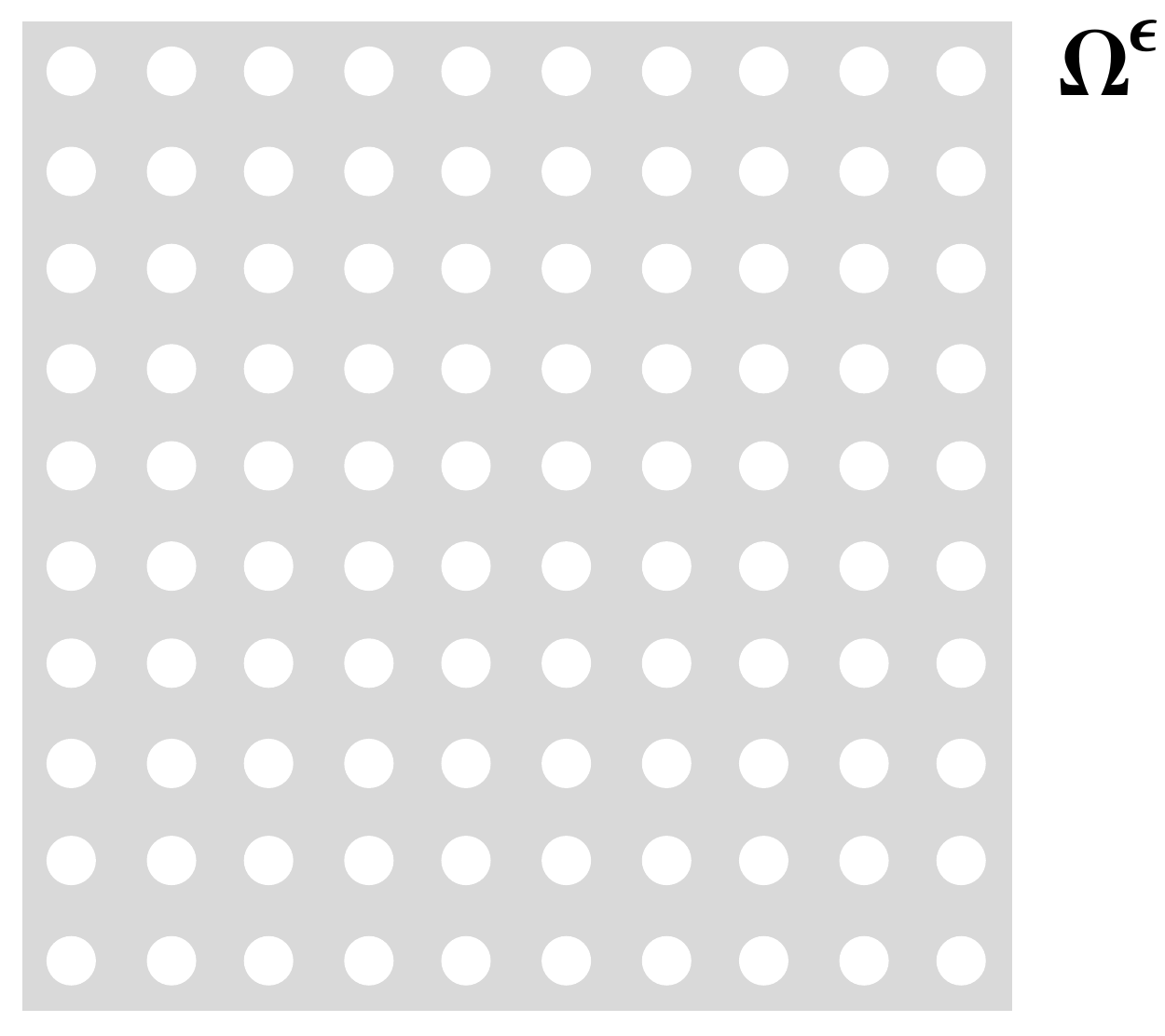}}
\caption{A periodic perforated domain $\Omega^\epsilon= (0,1)^2  \setminus \cup B_{r^\epsilon} (x_i)$.}
\label{fig1} 
\end{figure}

In \cite{CM} it is shown that there is a critical size of the holes (that is, a critical order of $r^\epsilon$ in $\epsilon$) such that
$$
v^\epsilon \to v^*, \qquad \mbox{ as } \epsilon \to 0,
$$
with $v^*$ given by
\begin{equation} \label{v^*.intro}
v^* = \left\{
\begin{array}{ll}
\mbox{the solution to } \Delta v^* = f , \qquad & \mbox{if } r^\epsilon \ll a^\epsilon, \\[4pt] 
\mbox{the solution to } \Delta v^* - \mu \, v^* = f , \qquad & \mbox{if } r^\epsilon = a^\epsilon, \\[4pt]
v^* = 0 , \qquad & \mbox{if } r^\epsilon \gg a^\epsilon, 
\end{array}
\right.
\end{equation}
with Dirichlet boundary conditions $v^*=0$ on $\partial \Omega$.
Assuming $N \geq 3$, we have  that the critical size of the holes is given by
$$
a^\epsilon \sim \epsilon^{\frac{N}{N-2}}.
$$

Note the extra term $\mu \, v^*$ that appears in the critical case. 
In this particular example, for $a^\epsilon = C_0 \epsilon^{\frac{N}{N-2}}$
with $C_0$ a constant, 
 $\mu$ is a positive constant that can be explicitly computed and is given by 
$$
\mu = \frac{S_N (N-2)}{2^N}C_0^{N-2}
$$
where $S_N$ is the surface of the sphere of radius one in $\R^N$.
Also, it is proved that the convergence $v^\epsilon \to v^*$ is weak in $H^1(\Omega)$ in the first two cases and 
strong when $r^\epsilon \gg a^\epsilon$.

For our nonlocal problem, in the same periodic setting, we get that the critical value
of $r^\epsilon$ is different from the local case and is given by
$$b^\epsilon = C_0 \, \epsilon$$ 
since in this case we obtain from Theorem \ref{theoD} that the limit $u^*$ verifies
$$
u^* = \left\{
\begin{array}{ll}
\displaystyle
\mbox{the solution to }  \int_{\R^N} J(x-y) (u^*(y) - u^*(x) )\,dy = f (x) , \qquad & \mbox{if } r^\epsilon \ll b^\epsilon, \\[9pt] 
\displaystyle 
\mbox{the solution to }  \int_{\R^N} J(x-y) (u^*(y) - u^*(x) )\,dy - \nu u^*(x) = f (x), \qquad & \mbox{if } r^\epsilon = b^\epsilon, \\[4pt]
u^* = 0 , \qquad & \mbox{if } r^\epsilon \gg b^\epsilon. \\[6pt]
\end{array}
\right.
$$
We have weak convergence in $L^2 (\Omega)$ in the first two cases and strong convergence in the last one.
Here the coefficient $\nu$ that appears in the critical case is also positive and can be explicitly computed. In fact,
$$\nu =  \frac{1 - \mathcal{X}}{\mathcal{X}}$$
where $\mathcal{X} \in L^\infty(\Omega)$ is just a positive constant, $\mathcal{X}=cte$,
determined by the proportion of the cube which is occupied by the hole. This follows from the fact that
in this periodic case we have
$\chi_\epsilon \rightharpoonup \mathcal{X} = |Q\setminus B| /|Q|$ (here $Q$ is the unit cube and $B$ is a ball 
of radius $C_0$ inside the cube). 
In some sense, the terms $\mathcal{X}$ and $(1 - \mathcal{X})$ in the limit problem can be seen as the effect of the holes in the original equation \eqref{1.1Dintro} and \eqref{BCDintro}.
The coefficient $\nu$ that appears in the critical case represents a kind of friction or drag caused by the perforations.

Concerning to the Neumann problem that we write as follows (see \cite{ElLibro,CERW}):
\begin{equation} \label{1.1intro}
f(x) = \int_{\R^N \setminus A^\epsilon} J (x-y) (u^\epsilon (y) - u^\epsilon (x)) dy, \qquad x  \in \Omega^\epsilon
\end{equation}
with 
\begin{equation} \label{BCNintro}
u^\epsilon (x) \equiv 0, \qquad x \in \R^N \setminus \Omega,
\end{equation}
and $f \in L^2(\Omega)$, where $A^\epsilon$ is the family of holes given by 
$$
A^\epsilon = \Omega \setminus \Omega^\epsilon.
$$
Note that we are integrating only in $\R^N \setminus A^\epsilon$ in the definition of our nonlocal operator, but still assume that 
$u^\epsilon \equiv 0$ in $\R^N \setminus \Omega$. Hence,
we are considering Neumann boundary conditions in the holes and a Dirichlet noundary condition
outside $\Omega$.
For this problem, we need to guarantee that the quantity 
$$
\lambda_1^\epsilon = \inf_{u\in W_\epsilon}  \frac{ \displaystyle
\frac12 \int_{\R^N \setminus A^\epsilon} 
\int_{\R^N\setminus A^\epsilon} J (x-y) (u (y) - u (x))^2 dy\, dx}{
\displaystyle \int_{\Omega^\epsilon} u^2(x) \, dx}
$$
with 
$$
W_\epsilon = \left\{ u\in
L^2({\R^N \setminus A^\epsilon}) \ : \ u (x) \equiv 0, \; x \in \R^N \setminus \Omega \right\}
$$
possesses a uniform lower bound. This is needed in order to obtain existence and uniqueness of solutions $u^\epsilon$ and is also necessary to study the asymptotic behavior of the problem as $\epsilon \to 0$.

We have the following result:

\begin{theorem} \label{theo1}
Let $\{ u^\epsilon \}_{\epsilon>0}$ be a family of solutions of problem \eqref{1.1intro} with \eqref{BCNintro}. 
Assume that the characteristic functions $\chi_\epsilon$ satisfy \eqref{CharF} and that
$\lambda_1^\epsilon \geq c>0$ with $c$ independent of $\epsilon$.
Then, there exists $u^* \in L^2(\Omega)$ such that 
$$\tilde u^\epsilon \rightharpoonup u^* \textrm{ weakly in } L^2(\Omega).$$

If $\mathcal{X} = 0$ in $L^\infty(\R^N)$ and 
$$ \int_{\R^N \setminus \Omega} J(x-y) \, dy
\geq m >0, \qquad \mbox{ for }x\in \Omega,$$ 
we have $u^*(x) = 0$ a.e. in $\R^N$.

If $\mathcal{X} \neq 0$ in $L^\infty(\R^N)$, the function $u^*$ satisfies the following nonlocal problem in $\Omega$,
$$
\mathcal{X}(x) \, f(x) = \mathcal{X}(x) \int_{\R^N}  J (x-y) (u^*(y) - u^*(x) ) \, dy -\Lambda(x) \, u^*(x) 
$$
with 
$$
u^*(x) \equiv 0, \qquad x \in \R^N \setminus \Omega,
$$
where $\Lambda \in L^\infty(\Omega)$ is given by 
$$
\Lambda(x) =  \int_{\R^N} J(x-y) \, ( 1- \chi_\Omega(y) + \mathcal{X}(y) ) \, dy - \mathcal{X}(x), \quad x \in \Omega.
$$
Here $\chi_\Omega$ is the characteristic function of the open set $\Omega$ and $\mathcal{X}$ is given by \eqref{CharF}.
\end{theorem}

Observe that to deal with the Neumann problem we need to assume extra conditions (besides \eqref{CharF})
on the sets $\Omega^\epsilon$, namely we need that  $\lambda_1^\epsilon \geq c>0$.
Concerning this assumption, $\lambda_1^\epsilon \geq c>0$, we will regard at $\lambda^\epsilon_1$ as the first eigenvalue of our Neumann problem.
Then, we will introduce a hypothesis involving the geometry of $\Omega^\epsilon$ and the kernel $J$ 
(see condition ${\bf (H_N)}$ 
in Section \ref{sect-Neumann}) which ensures the validity of $\lambda^\epsilon_1 \geq c$. We also include a simple example that shows that in general it could happend that $\lambda^\epsilon_1 =0$ (in this case we do not have existence of solutions to our nonlocal Neumann problem for a general datum $f$).
In Section \ref{sect-perodic}, we verify that this assumption ${\bf (H_N)}$ holds in the classical case of periodic perforated domains. This fact also allows us to obtain the limit equation to the Neumann problem \eqref{1.1intro} with \eqref{BCNintro} in the case of a periodically perforated domain.

Note the more involved term $\Lambda$ that appears in Theorem \ref{theo1}.
Rewriting it as 
$$
\Lambda(x) =  \int_{\R^N \setminus \Omega} J(x-y) \, dy + \int_{\R^N} J(x-y) 
\left( \mathcal{X}(y) - \mathcal{X}(x) \right) dy, \quad x \in \Omega, 
$$
we see that the kernel $J$ explicitly affects the extra term in the limit problem to the critical case for the Neumann problem.
This dependence of the extra term on the kernel $J$ does not occur in the Dirichlet problem where the coefficient $\nu$ only depends on the perturbation of the domain via $\mathcal{X}$.

Many techniques and methods have been developed in order to understand the effect of the holes in perforated domains on the solutions of PDE problems with different boundary values.
From pioneering works to recent ones we can still mention \cite{ GA, ADMET, CJM, DAPGR, IP,N, N2, ESP, MTM} and references therein that are
concerned with elliptic and parabolic equations, nonlinear operators, as well as Stokes and Navier-Stokes equations from fluid mechanics. 
Note that this kind of problem is an ``homogenization" problem, since the heterogeneous domain $\Omega^\epsilon$ is replaced by a homogeneous one, $\Omega$, in the limit. However, up to our knowledge, this is the first 
paper to deal with this kind of homogenization problem for a nonlocal operator with a non-singular kernel.

For homogenization results for singular kernels we refer to \cite{Ca,sw2,Wa} (we emphasize that those references
deal with homogenization in the coefficients involved in the equation and not with perforated domains as it is the case here). For random homogenization of an obstacle problem we refer to \cite{caffa2}. 
We also remark that the case of stochastic homogenization (this is, the case in which the holes
are randomly distributed inside $\Omega$) is not treated here.

On the other hand, nonlocal equations with non-singular kernel attracted some attention 
recently, see \cite{ElLibro,chfrt,Du,Ignat,Ignat2,RS} for a non-exhaustive list of references. 
We also mention \cite{JDE,PtAn} where asymptotic problems in such nonlocal equations have been recently studied.
Besides the applied models with such kernels (for example, we refer to elasticity models, \cite{Per}), the mathematical interest is mainly due to the fact that, in general, there is no regularizing effect and 
therefore no general compactness tools are available.

Now, we comment briefly on our hypothesis and results. 
First, 
we remark that we only obtain weak convergence in $L^2$ of the solutions $u^\epsilon$. This is due to the fact that
the nonlocal operator does not regularize (and hence solutions $u^\epsilon$ are expected to be bounded
in $L^2$ but nothing better) and is analogous to the fact that for the usual local case we have weak convergence in $H^1$.

On the other hand, we note that our results are valid under very general assumptions on the perturbed domains.
Namely we only require that the characteristic functions of the involved domains converge weakly. Of course, this
is verified in the periodic case that is our leading example (but we are not restricted to this case).

Indeed, we can allow many other different situations than perforated domains.
For instance, we can consider here a family of domains $\Omega^\epsilon$ whose the boundary presents a highly oscillatory behavior as the parameter $\epsilon \to 0$. 
We take as a fixed domain the rectangle $\Omega = (0,1) \times (-1,1)$, and as a family of perturbed domains 
$$
\Omega^\epsilon = \{ (x,y) \in \R^2 \, : \, x \in (0,1), \; -1 < y < 0.5 \, ( 1+  \sin (x/\epsilon) ) \}.
$$
We illustrate this simple situation in Figure \ref{figosc} below.  
Note that here, the period and amplitude of the oscillations are the same order with respect to the positive parameter $\epsilon$. 
Then, if $\chi$ is the characteristic function of the open set 
$$
Y = \{ (z, w) \in \R^2 \, : \, z \in (0, 2 \pi), \; -1 < w < 0.5 \, ( 1 + \sin z ) \},
$$
we have that the characteristic function $\chi_\epsilon$ of $\Omega^\epsilon$ is given by
$
\chi_\epsilon(x,y) = \hat \chi (x/\epsilon, y)
$
where $\hat \chi$ is the periodic extension of $\chi$ with respect to its first variable (that is, in the horizontal direction). 
Thus, we get from Averaging Theorem for oscillating functions that 
$$
\chi_\epsilon \rightharpoonup \mathcal{X} = \frac{1}{2 \pi} \int_0^{2 \pi} \hat \chi(s,\cdot) \, ds  \quad  \textrm{ weakly$^*$ in } L^\infty(\Omega)
$$
as $\epsilon \to 0$.
Since $\mathcal{X} \neq \chi_\Omega$ in general, we obtain from Theorems \ref{theoD} and \ref{theo1} 
a non trivial nonlocal limit problem to this case.

We quote here \cite{AB, BPP, CK, MSZ, NPS} and references therein where local problems to partial differential equations in highly oscillating domains 
are deeply studied. 

\begin{figure}[htp] 
\centering \scalebox{0.38}{\includegraphics{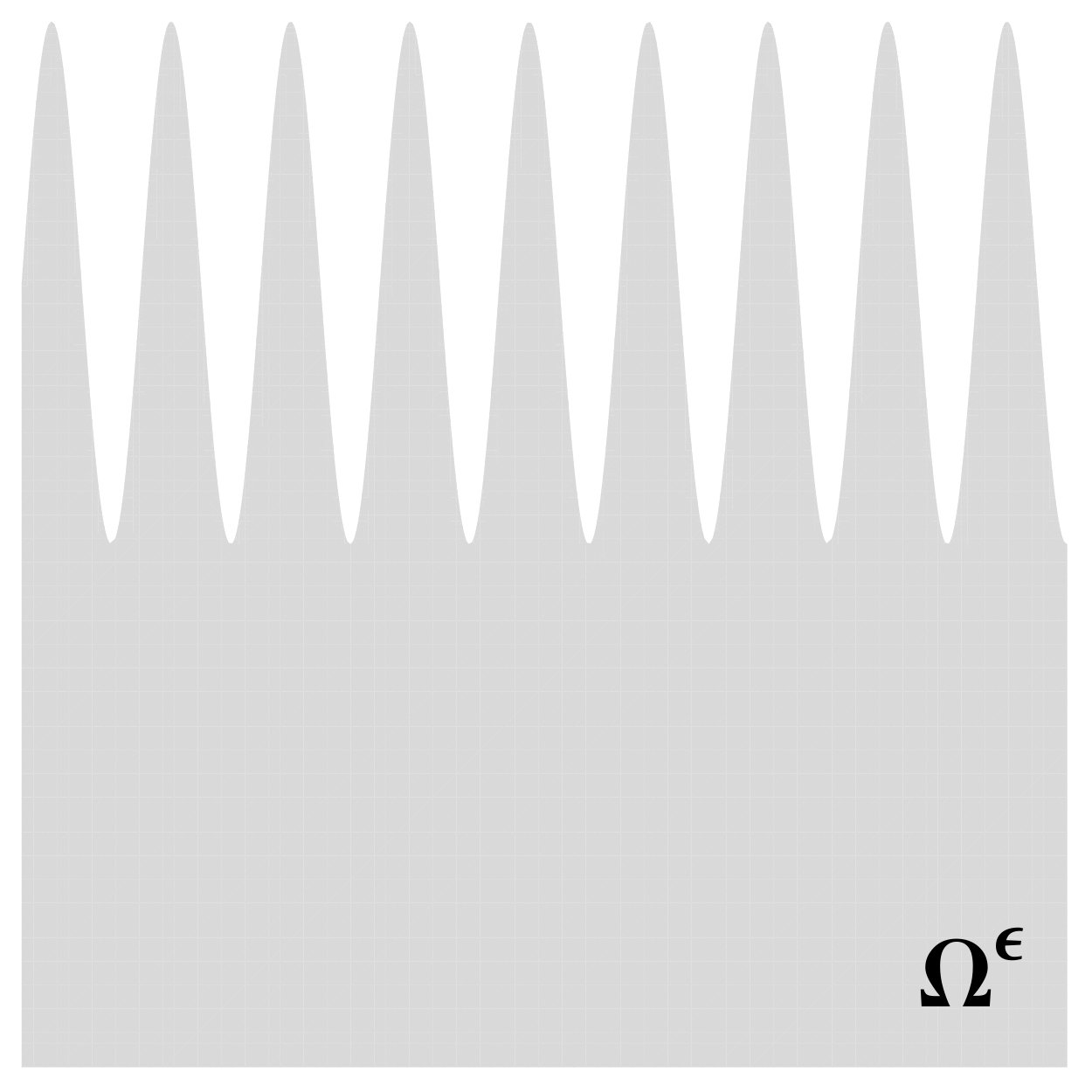}}
\caption{A domain $\Omega^\epsilon$ with oscillating boundary.}
\label{figosc} 
\end{figure}

Finally, our aim is to see how these problems behave when we introduce another parameter that
controls the size of the nonlocality.
In \cite{ElLibro} (see also \cite{CER1,CERW}) it is shown that we can obtain solutions to local problems as limits 
of solutions to nonlocal problems when we rescale the kernel considering
$$
J_\delta (z) = \frac{C}{\delta^{N+2}} J(\frac{z}{\delta})
$$
and letting $\delta \to 0$. Here $C= 2(\int_{\R^N} J(z) z_1^2)^{-1}$ is just a normalizing constant.
If we apply this idea to our nonlocal problem in the Dirichlet case we are lead to consider
\begin{equation} \label{Dir-ep-delsec.intro}
f(x) =\frac{C}{\delta^{N+2}} \int_{\R^N} J \left(\frac{x-y}{\delta}\right) 
(u^{\epsilon,\delta} (y) - u^{\epsilon,\delta} (x)) dy, \qquad x \in \Omega^\epsilon,
\end{equation}
with $u^{\epsilon,\delta} (x) \equiv 0$, for $x \in \R^N \setminus \Omega^\epsilon$.

Our aim is to study the limits $\epsilon \to 0$
(to have an homogenized limit problem) and $\delta \to 0$ (to approach local problems).
To perform this analysis, we restrict ourselves to the periodic case for perforated domains.  
That is, we consider $\Omega^\epsilon = \Omega \setminus \cup B_{r^\epsilon} (x_i)$
where $B_{r^\epsilon} (x_i)$ is a ball strictly contained in $\Omega$ centered 
in $x_i\in \Omega$ of the form $x_i \in 2 \epsilon \Z^N$ with $0 < r^\epsilon < \epsilon \leq 1$.

We show that, for $u^{\epsilon, \delta}$ we have that the iterated limits
$$
 \lim_{\epsilon \to 0} \lim_{\delta \to 0} \tilde{u}^{\epsilon, \delta} = v, \qquad \mbox{and}
 \qquad \lim_{\delta \to 0} \lim_{\epsilon \to 0} \tilde{u}^{\epsilon, \delta} = w,
$$
exist, but, in general, they do not commute, that is, in general $w \neq v$.
Here, $v$ is given by \eqref{v^*.intro}, that is,
$$
v = \left\{
\begin{array}{ll}
\mbox{the solution to } \Delta v = f , \qquad & \mbox{if } r^\epsilon \ll a^\epsilon, \\[4pt] 
\mbox{the solution to } \Delta v - \mu \, v = f , \qquad & \mbox{if } r^\epsilon = a^\epsilon, \\[4pt]
v = 0 , \qquad & \mbox{if } r^\epsilon \gg a^\epsilon, 
\end{array}
\right.
$$
with Dirichlet boundary conditions, $v=0$ on $\partial \Omega$,
and $w$ by 
\begin{equation}
\label{chichi.intro}
w = \left\{
\begin{array}{ll}
\mbox{the solution to } \Delta w = f \mbox{ with $w=0$ on $\partial \Omega$}
, \qquad & \mbox{if } r^\epsilon \ll b^\epsilon, \\[4pt] 
0, \qquad & \mbox{if } r^\epsilon = b^\epsilon.
\end{array}
\right.
\end{equation}

Note that $v$ is the limit for the local problem, in fact, when we compute first
the limit as $\delta \to 0$ of $\tilde{u}^{\epsilon, \delta} $ we obtain a solution to a local problem 
with the Laplacian, then the limit $\lim_{\epsilon \to 0} \lim_{\delta \to 0} \tilde{u}^{\epsilon, \delta} $
coincides with the one that holds in the local case. On the other hand, when we first take the limit as $\epsilon \to 0$
from our results we get a solution to a nonlocal problem (with a different size of the critical radius) and when we localize this problem letting $\delta \to 0$ we obtain $w$ a solution to a local problem but different from the previous one (in general). We remark that in the limit $\lim_{\delta \to 0} \lim_{\epsilon \to 0} \tilde{u}^{\epsilon, \delta}$ we have weak convergence in $L^2$ while for $\lim_{\epsilon \to 0} \lim_{\delta \to 0} \tilde{u}^{\epsilon, \delta} $
the convergence is strong in $L^2$.

A similar situation (rescaling the kernel with a parameter $\delta$ in a periodically perforated domain)
can be studied for the Neumann case.  
Now we consider
\begin{equation} \label{Neu-ep-delsec.intro}
f(x) =\frac{C}{\delta^{N+2}} \int_{\R^N \setminus A^\epsilon} J \left(\frac{x-y}{\delta}\right) 
(u^{\epsilon,\delta} (y) - u^{\epsilon,\delta} (x)) dy, \qquad x \in \Omega^\epsilon,
\end{equation}
with $
u^{\epsilon,\delta} (x) \equiv 0$, for $x \in \R^N \setminus \Omega$.
In this case we have
$$
\lim_{\epsilon \to 0} P_\epsilon \left( \lim_{\delta \to 0} u^{\delta,\epsilon} \right) = v \quad \textrm{ in } L^2(\Omega),
$$
with $P_\epsilon$ an extension operator. Here the limit $v$ is given by
\begin{equation}
\label{chi-Neu.intro}
v = \left\{
\begin{array}{ll}
\mbox{the solution to } \Delta v = f , \qquad & \mbox{if } r^\epsilon \ll b^\epsilon, \\[6pt] 
\displaystyle \mbox{the solution to } \sum_{i,j=1}^N  q_{ij} \frac{\partial^2 v}{\partial x_i \partial x_i} = \frac{|Q \setminus B|}{|Q|}f , \qquad & \mbox{if } r^\epsilon = b^\epsilon,
\end{array}
\right.
\end{equation}
with Dirichlet boundary conditions, $v=0$ on $\partial \Omega$.
The constants $q_{ij}$ are the \emph{homogenized coefficients} and can be explicitly 
computed (see \cite{CioS} and Section \ref{sect-rescales}).

On the other hand, the limit 
$$
\lim_{\delta \to 0} \left( \lim_{\epsilon \to 0} \tilde{u}^{\delta,\epsilon} \right) = w 
$$
exists (but this time the convergence is weak in $L^2(\Omega)$), and is given by
\begin{equation}
\label{chichi-Neu.intro}
w = \left\{
\begin{array}{ll}
\mbox{the solution to } \Delta w = f , \mbox{ with $w=0$ on $\partial \Omega$},\qquad & \mbox{if } r^\epsilon \ll b^\epsilon, \\[6pt] 
0 , \qquad & \mbox{if } r^\epsilon = b^\epsilon.
\end{array}
\right.
\end{equation}

The paper is organized as follows: in Section \ref{sect-Dirichlet} we deal with the Dirichlet
problem while in Section \ref{sect-Neumann} we consider the Neumann case. 
In Section \ref{sect-perodic} we deal with the case of periodically distributed holes. 
Finally, in Section \ref{sect-rescales} we rescale the kernel.

\section{The Dirichlet problem.} \label{sect-Dirichlet}
\setcounter{equation}{0}

We deal here with the nonlocal Dirichlet problem. For $f \in L^2(\Omega)$, we consider 
\begin{equation} \label{1.1.Dir}
f(x) = \int_{\R^N } J (x-y) (u^\epsilon (y) - u^\epsilon (x)) dy, \qquad x  \in \Omega^\epsilon
\end{equation}
with 
\begin{equation} \label{Cond.Dir}
u^\epsilon (x) \equiv 0, \qquad x \in \R^N \setminus \Omega^\epsilon.
\end{equation}

Observe that existence and uniqueness of our problem
follows considering the variational problem
$$
\min_{u \in W_\epsilon} \frac14 \int_{\R^N} \int_{\R^N} J (x-y) (u (y) - u (x))^2 dy\, dx -
\int_{\R^N} f(x) u (x) \, dx
$$
with
$$
W_\epsilon = \left\{ u \in L^2(\Omega^\epsilon) \quad : \quad u^\epsilon (x) \equiv 0, \; x \in \R^N \setminus \Omega^\epsilon \right\}.
$$
It follows from (${\bf H_J}$) that the unique minimizer (that we call $u^\epsilon$) is a solution to \eqref{1.1.Dir} and \eqref{Cond.Dir} since it holds that
\begin{equation} \label{1.1.Dweak}
\begin{array}{l} 
\displaystyle 0   =  - \frac12 \int_{\R^N} \int_{\R^N} J(x-y) (u^\epsilon(y) - u^\epsilon(x) ) (\varphi(y) - \varphi(x) ) \, dy dx 
- \int_{\R^N}  f(x) \varphi (x) \, dx     \\[10pt]
 \qquad \displaystyle =   \int_{\R^N} \varphi(x) \int_{\R^N} J (x-y) (u^\epsilon (y) - u^\epsilon (x)) \, dy dx
  - \int_{\R^N}  f(x) \varphi (x) \, dx    
\end{array}
\end{equation}
for any $\varphi \in L^2 (\R^N)$ with $\varphi(x) \equiv 0$ for $x \in \R^N \setminus \Omega^\epsilon$. 
Note that, if $u^\epsilon \in L^2(\Omega^\epsilon)$ is a solution to \eqref{1.1.Dir} and \eqref{Cond.Dir}, then
it is the unique minimizer.
Taking $\varphi =u^\epsilon$ in \eqref{1.1.Dweak} we get
\begin{equation} \label{DirEig}
\begin{array}{l}
\displaystyle - \int_{\Omega^\epsilon}  f(x) u^\epsilon (x) \, dx  \displaystyle 
=  \frac12
  \int_{\R^N}  \int_{\R^N} J (x-y) (u^\epsilon (y) - u^\epsilon (x))^2\, dy dx 
 \geq \beta_1^\epsilon  \int_{\Omega^\epsilon} (u^\epsilon(x))^2 \, dx
  \end{array}
\end{equation}
where $\beta_1^\epsilon$ is the first eigenvalue associated with this operator in the space
$W_\epsilon$. It is given by
\begin{equation} \label{eigenD}
\beta_1^\epsilon = \inf_{u\in W_\epsilon}  \frac{ \displaystyle
\frac12 \int_{\R^N} \int_{\R^N} J (x-y) (u (y) - u (x))^2 dy\, dx}{
\displaystyle \int_{\Omega^\epsilon} u^2(x) \, dx}.
\end{equation}
From \cite[Proposition 2.3]{ElLibro} we know that $\beta_1^\epsilon$ is strictly positive.
Therefore, due to \eqref{DirEig}, we get 
$$
\| u^\epsilon \|_{L^2(\Omega^\epsilon)}  \leq \frac 1{\beta_1^\epsilon} \| f \|_{L^2(\Omega)}. 
$$
Thus, since $\beta_1^\epsilon \geq c>0$ with $c$ independent of $\epsilon$ (see Lemma \ref{LEBD} below), we also obtain 
$$\| u^\epsilon \|_{L^2(\Omega^\epsilon)} \leq K$$
for some positive constant $K$ depending only on $f$ (and so, independent of $\epsilon$). 
Then, along a subsequence if necessary, 
\begin{equation} \label{weakLD}
\tilde u^\epsilon \rightharpoonup u^* \quad \textrm{ weakly in } L^2 (\Omega)
\end{equation}
as $\epsilon \to 0$ where $\, \tilde \cdot \,$ denotes the extension by zero applied to functions defined in subsets of $\R^N$.

Thus, if $\chi_\epsilon$ is the characteristic function of $\Omega^\epsilon$ and $\tilde u^\epsilon$ is the extension by zero of  
$u^\epsilon$ to $\R^N$, we can use $\int_{\R^N} J(x-y) \, dy=1$ to rewrite \eqref{1.1.Dweak} 
\begin{equation} \label{1.2.D}
\begin{array}{l}
 \displaystyle\int_{\Omega}  \chi_\epsilon(x) \, \varphi (x) f(x) \, dx  =  \int_{\Omega}  \chi_\epsilon(x) \, \varphi (x)  
 \left( \int_{\R^N} J(x-y) \, \tilde u^\epsilon (y) \, dy \right) dx \displaystyle 
  - 
  \int_{\Omega}  \varphi (x) \, \tilde u^\epsilon(x) \, dx
  \end{array}
\end{equation}
for any $\varphi \in L^2(\Omega)$.

%

Now we are ready to proceed with the proof of Theorem \ref{theoD}.

\begin{proof}[Proof of Theorem \ref{theoD}]
We need to pass to the limit in \eqref{1.2.D}.
In order to do that, we have to evaluate 
\begin{equation} \label{Ufunction}
U_\epsilon(x) = \int_{\R^N} J(x-y) \tilde u^\epsilon(y) \, dy 
\end{equation}
which is defined for any $x \in \R^N$.
From \eqref{weakLD}, we have that  
$$
\int_{\R^N} J(x-y) \tilde u^\epsilon(y) \, dy = \int_\Omega J(x-y) \tilde u^\epsilon(y) \, dy 
\to \int_\Omega J(x-y) \, u^*(y) \, dy
$$
as $\epsilon \to 0$ for each $x \in \R^N$, and then, 
\begin{equation} \label{eqU1}
U_\epsilon(x) \to U_0(x) = \int_\Omega J(x-y) \, u^*(y) \, dy, \qquad \textrm{ for all } x \in \R^N.
\end{equation}
Hence, since $u^\epsilon$ is uniformly bounded in $L^2(\Omega^\epsilon)$ and  
\begin{equation} \label{eqU2}
|U_\epsilon(x)| =  \left| \int_\Omega J(x-y) \, \tilde u^\epsilon (y) \, dy \right| 
 \leq  |\Omega|^{1/2} \| J \|_\infty \| u^\epsilon \|_{L^2(\Omega^\epsilon)},
\end{equation}
we obtain by Dominated Convergence Theorem that
$$
\int_{\Omega} U^\epsilon(x) \, \varphi(x) \, dx \to \int_\Omega U_0(x) \, \varphi(x) \, dx 
$$
for all $\varphi \in L^2(\Omega)$, and then, $U^\epsilon \rightharpoonup U_0$  weakly in $L^2(\Omega)$.
Indeed, we can prove that 
\begin{equation} \label{UL2C}
U^\epsilon \to U_0 \quad \textrm{ strongly in } L^2(\Omega). 
\end{equation}
It follows from \eqref{eqU1} and \eqref{eqU2} that 
$$
(U_\epsilon(x))^2 \to (U_0(x))^2 \quad \textrm{ for all } x \in \R^N 
$$
with 
$$
|U_\epsilon(x)|^2 \leq  |\Omega| \| J \|^2_\infty \| u^\epsilon \|^2_{L^2(\Omega^\epsilon)}.
$$
Consequently, using again the Dominated Convergence Theorem, we get 
$$
\| U^\epsilon \|_{L^2(\Omega)} \to \|U_0\|_{L^2(\Omega)}
$$
as $\epsilon \to 0$ proving \eqref{UL2C} since we are working in Hilbert spaces.

We now can combine \eqref{CharF}, \eqref{weakLD}, \eqref{1.2.D} and \eqref{Ufunction} to obtain 
\begin{equation} \label{1.3.D}
\begin{array}{l}
 \displaystyle\int_{\Omega}  \mathcal{X}(x) \, \varphi (x) f(x) \, dx  =  \int_{\Omega}  \mathcal{X}(x) \, \varphi (x)  
 \left( \int_{\R^N} J(x-y) \, u^* (y) \, dy \right) dx \displaystyle 
  - 
  \int_{\Omega}  \varphi (x) \, u^*(x) \, dx
  \end{array}
\end{equation}
for any $\varphi \in L^2(\Omega)$.
By the addition and subtraction of the term 
$$
\int_\Omega \varphi(x) \, \mathcal{X}(x) \, u^*(x) \, dx = \int_\Omega \varphi(x) \, \mathcal{X}(x) \int_{\R^N} J(x-y) \, u^*(x) \, dy \, dx
$$
we still can rewrite \eqref{1.3.D} as 
\begin{equation} \label{1.4.D}
\begin{array}{l}
 \displaystyle\int_{\Omega}  \mathcal{X}(x) \, \varphi (x) f(x) \, dx  =  \int_{\Omega}  \mathcal{X}(x) \, \varphi (x)  
 \left( \int_{\R^N} J(x-y) \, ( u^* (y) - u^*(x) ) \, dy \right) dx \displaystyle 
\\[10pt]
  \qquad \qquad \qquad  \qquad  \qquad  \qquad \displaystyle 
  - 
  \int_{\Omega}  \varphi (x) \, u^*(x) \, (1 - \mathcal{X}(x)) \, dx
  \end{array}
\end{equation}
which implies 
\begin{equation} \label{1.1.a.e.}
\mathcal{X}(x) \, f(x) = \mathcal{X}(x) \int_{\R^N}  J (x-y) (u^*(y) - u^*(x) ) \, dy - (1 - \mathcal{X}(x)) \, u^*(x), \quad a.e. \;  \Omega, 
\end{equation}
with 
$u^*(x) \equiv 0$, $x \in \R^N \setminus \Omega$.

Finally, we show that $u^*$ is the unique solution of \eqref{1.1.a.e.}. This fact implies the convergence of whole sequence $\tilde u^\epsilon$.
To do that, we consider the set $D$ where the function $\mathcal{X}$ vanishes  
\begin{equation} \label{dominioD}
D = \{ x \in \Omega \, : \, \mathcal{X}(x) = 0 \}.
\end{equation}
From \eqref{1.1.a.e.}, we have $u^*(x) = 0$ for all $x \in D$.
Hence, if $D = \Omega$ a.e., $u^*(x)=0$ a.e. $x \in \R^N$ is the unique solution.
Thus, let us consider that $\Omega \setminus D$ is a nontrivial measurable set.
We can rewrite equation \eqref{1.1.a.e.} as 
\begin{equation} \label{a.e.D1}
f(x) = \int_{\R^N}  J (x-y) (u^*(y) - u^*(x) ) \, dy - \frac{(1 - \mathcal{X}(x))}{\mathcal{X}(x)} \, u^*(x), \quad a.e. \;  \Omega \setminus D, 
\end{equation}
with 
$u^*(x) \equiv 0$, $x \in \R^N \setminus (\Omega \setminus D)$.
Hence, if $u^*$ and $v^*$ are solutions of \eqref{a.e.D1}, $w^* = u^* - v^*$ satisfies 
\begin{equation} \label{a.e.D2}
0 = \int_{\R^N}  J (x-y) (w^*(y) - w^*(x) ) \, dy - \frac{(1 - \mathcal{X}(x))}{\mathcal{X}(x)} \, w^*(x) , \quad a.e. \;  \Omega \setminus D. 
\end{equation}
Consequently, $\frac{(1 - \mathcal{X})}{\mathcal{X}} \, w^* \in L^2(\Omega \setminus D)$, and then, 
\begin{eqnarray*}
\int_{\Omega \setminus D} \frac{1 - \mathcal{X}(x)}{\mathcal{X}(x)} \, (w^*(x))^2 \, dx 
 =  - \frac 12 \int_{\R^N} \int_{\R^N} J (x-y) (w^*(y) - w^*(x) )^2 \, dy \, dx 
 \leq  0.  
\end{eqnarray*}
Since $\frac{1 - \mathcal{X}(x)}{\mathcal{X}(x)} \geq 0$ in $\Omega \setminus D$, we get 
$$
\frac{1 - \mathcal{X}(x)}{\mathcal{X}(x)} \, (w^*(x))^2 = 0 \quad \textrm{ a.e. } x \in \Omega \setminus D.
$$ 
This combined with \cite[Proposition 2.2]{ElLibro} completes the proof.
\end{proof}

If we still assume in the hypotheses of Theorem \ref{theoD} that the characteristic functions $\chi_\epsilon$ converges weakly star to 
$\mathcal{X}(x) = 0$ a.e. in $\Omega$, we get that the limit function $u^*$ must be null with strong convergence in $L^2$. 

\begin{cor} \label{corD}
Under the hypotheses of Theorem \ref{theoD} with $\mathcal{X}(x) = 0$ a.e. $x \in \Omega$, we obtain 
$u^*(x) = 0$ a.e. $x \in \R^N$ 
with  
$$
\tilde u^\epsilon \to 0, \quad \textrm{ strongly in } L^2(\Omega).
$$
\end{cor}
\begin{proof}
From the limit equation in Theorem \ref{theoD} is very easy to see that $u^* = 0$. 
We obtain strong convergence just observing that the norm of $\tilde u^\epsilon$ also converges to $0$ as $\epsilon \to 0$.
Indeed, from \eqref{1.2.D} with $\varphi = \tilde u^\epsilon$ we have 
\begin{eqnarray*}
\| \tilde u^\epsilon \|_{L^2(\Omega)}^2 & = & \int_{\Omega} \tilde u^\epsilon(x) \int_{\R^N} J(x-y) \, \tilde u^\epsilon(y) \, dy - \int_\Omega \tilde u^\epsilon(x) \, f(x) \, dx \\
& \to & \int_{\Omega} u^*(x) \int_{\R^N} J(x-y) \, u^*(y) \, dy - \int_\Omega u^*(x) \, f(x) \, dx = 0.
\end{eqnarray*}
\end{proof}

\begin{remark} {\rm
It is not difficult to see that a more general situation can be considered in \eqref{1.1.Dir} and \eqref{Cond.Dir}.
In a similar way we can deal with
$$
f^\epsilon(x) = \int_{\R^N} J(x-y) (u^\epsilon(y) - u^\epsilon(x)) \, dy
$$ 
assuming $f^\epsilon \to f$ strongly in $L^2(\Omega)$ since 
$\|f^\epsilon\|_{L^2(\Omega)}$ remains uniformly bounded in $\epsilon$ and 
$$
\int_\Omega \chi_\epsilon(x) \, f^\epsilon(x) \, dx \to \int_\Omega \mathcal{X}(x) \, f(x) \, dx
$$
as $\epsilon \to 0$.}
\end{remark}

Now we finish the section showing that the sequence of first eigenvalues converges to a positive value. Therefore, this sequence possesses a positive lower bound.
\begin{lemma} \label{LEBD}
Let $\beta^\epsilon_1$ be the family of first eigenvalues introduced in \eqref{eigenD}.
Then, there exist $\beta_1>0$, such that 
$$
\beta^\epsilon_1 \to \beta_1, \quad \textrm{ as } \epsilon \to 0.
$$ 
Consequently, there exist $\epsilon_0>0$ and $c>0$ with 
$$
\beta^\epsilon_1 > c >0, \quad \forall \epsilon \in (0,\epsilon_0). 
$$
\end{lemma}
\begin{proof}
We start observing that $\beta^\epsilon_1$ is a bounded sequence with respect to $\epsilon$.
Indeed, it follows from \cite[Proposition 2.3]{ElLibro} that 
$$
0<\beta^\epsilon_1<1.
$$
Thus, we can extract a subsequence, still denoted by $\beta_1^\epsilon$, such that $\beta_1^\epsilon \to \beta_1$.
Since $\beta_1^\epsilon$ is strictly positive, $\beta_1$ is non negative.

Now let $\phi^\epsilon$ be the associated eigenfunction to $\beta_1^\epsilon$ with $\| \phi^\epsilon \|_{L^2(\Omega^\epsilon)} =1$ (the existence of such eigenfunction can be proved as in \cite{jorge}). 
Hence,  
$$
- \beta_1^\epsilon \int_{\Omega^\epsilon} \phi^\epsilon(x) \, \varphi(x) \, dx 
= \int_{\Omega^\epsilon} \varphi(x) \int_{\R^N} J(x-y) ( \phi^\epsilon(y) - \phi^\epsilon(x) ) \, dy dx
$$
for any $\varphi \in L^2(\Omega)$ with $\phi^\epsilon$ vanishing in $\R^N \setminus \Omega^\epsilon$.
Therefore, using the extension by zero to whole space and the characteristic function $\chi_\epsilon$ of $\Omega^\epsilon$, we get 
\begin{equation} \label{eqeigenD1}
\begin{array}{l}
 \displaystyle - \beta_1^\epsilon \int_{\Omega} \tilde \phi^\epsilon(x) \, \varphi(x) \, dx 
= \int_{\Omega} \chi_\epsilon(x) \, \varphi(x) \int_{\R^N} J(x-y) \, \tilde \phi^\epsilon(y) \, dy dx \displaystyle 
\\[10pt]
  \qquad \qquad \displaystyle 
 - \int_{\Omega} \varphi(x) \, \tilde \phi^\epsilon(x) \int_{\R^N} J(x-y) \, dy dx.
\end{array}
\end{equation}
Due to $\| \phi^\epsilon \|_{L^2(\Omega^\epsilon)} =1$, we can extract a subsequence, still denoted by $\tilde \phi^\epsilon$, such that 
\begin{equation} \label{phi*}
\tilde \phi^\epsilon \rightharpoonup \phi^* \quad \textrm{ weakly in } L^2(\Omega)
\end{equation}
with $\phi^* \equiv 0$ as $x \in \R^N \setminus \Omega$. 

First, let us discuss the case $\phi^* = 0$. 
From \eqref{eqeigenD1} with $\varphi = \tilde \phi^\epsilon$, we have that 
$$
  (1- \beta_1^\epsilon) \int_{\Omega} {\tilde \phi^\epsilon(x)}^2 \, dx 
= \int_{\Omega} \tilde \phi^\epsilon(x) \int_{\R^N} J(x-y) \, \tilde \phi^\epsilon(y) \, dy dx,  
$$
and then,
\begin{equation} \label{nozero}
  (1- \beta_1^\epsilon) = \int_{\Omega} \tilde \phi^\epsilon(x) \int_{\R^N} J(x-y) \, \tilde \phi^\epsilon(y) \, dy dx 
\end{equation}
since $\| \phi^\epsilon \|_{L^2(\Omega^\epsilon)} =1$. 
Arguing as in \eqref{Ufunction}, we can show from assumptions \eqref{phi*} and the fact that 
$\phi^*(x) = 0$ a.e. in $\R^N$ that 
$$
\int_{\R^N} J(\cdot-y) \, \tilde \phi^\epsilon(y) \, dy  \to 0 \quad \textrm{ strongly in } L^2(\Omega),
$$
as $\epsilon \to 0$. 
Hence, it follows from \eqref{phi*} and \eqref{nozero} that $\beta_1^\epsilon \to 1$ proving our result.
Therefore, we can suppose $\phi^* \neq 0$ in $L^2(\Omega)$.

Now, we can argue as in \eqref{1.3.D} to pass to the limit in \eqref{eqeigenD1} obtaining the following limit equation
\begin{equation} 
\begin{array}{l}
 \displaystyle - \beta_1 \int_{\Omega} \phi^*(x) \, \varphi(x) \, dx 
 \displaystyle = \int_{\Omega} \mathcal{X}(x) \, \varphi(x) \int_{\R^N} J(x-y) \,  \phi^*(y) \, dy dx  
 - \int_{\Omega} \varphi(x) \, \phi^*(x) \, dx \\[10pt]
 \qquad  = \displaystyle \int_\Omega  \mathcal{X}(x) \, \varphi(x) \int_{\R^N} J(x-y) \, ( \phi^*(y) - \phi^*(x) ) \, dy \, dx \displaystyle 
 - \int_\Omega (1-\mathcal{X}(x)) \, \phi^*(x) \, \varphi(x) \, dx,
\end{array}
\end{equation}
which can be rewritten as 
\begin{equation} \label{eigen.a.e.}
-\beta_1 \, \phi^*(x) = \mathcal{X}(x) \int_{\R^N}  J (x-y) (\phi^*(y) - \phi^*(x) ) \, dy - (1 - \mathcal{X}(x)) \, \phi^*(x), \quad a.e. \; \Omega, 
\end{equation}
with 
$\phi^*(x) \equiv 0$, $x \in \R^N \setminus \Omega$.

If $\mathcal{X}(x) \equiv 1$ a.e. $\Omega$, 
it follows from \eqref{eigen.a.e.} that $\beta_1$ is the first eigenvalue of the self-adjoint operator $T:L^2(\Omega) \mapsto L^2(\Omega)$ given by
$$
T(\phi) = \int_{\R^N}  J (x-y) (\phi(y) - \phi(x) ) \, dy.
$$
Hence, from \cite[Proposition 2.3]{ElLibro} we get $\beta_1>0$, and the proof is completed. 

Thus, let us assume $\mathcal{X} \neq 1$ in $L^\infty(\Omega)$, and consider the set
 $D$ given by the vanishing points of $\mathcal{X}$ introduced in \eqref{dominioD}.
From \eqref{eigen.a.e.}, we get $\phi^*(x) \, (1 - \beta_1) = 0$ for all $x \in D$.
If $\beta_1 =1$, the proof is complete. If it is not the case, we have 
$\phi^*(x) \equiv 0$ in $x \in D$. Let us suppose $\phi^*(x) \equiv 0$ in $D$.
In fact, without loss of generality, we can assume $\phi^*(x) \neq 0$ in $\Omega \setminus D$.
From \eqref{eigen.a.e.} we obtain
\begin{equation} \label{eigen.a.e.2}
\frac{1 - \mathcal{X}(x) - \beta_1}{\mathcal{X}(x)} \, \phi^*(x)  = \int_{\R^N}  J (x-y) (\phi^*(y) - \phi^*(x) ) \, dy, \quad a.e. \; \Omega \setminus D, 
\end{equation}
with 
$\phi^*(x) \equiv 0$ wherever  $x \in \R^N \setminus (\Omega \setminus D)$.
Since $\phi^* \in L^2(\Omega)$, it follows from \eqref{eigen.a.e.2} that $\frac{1 - \mathcal{X} - \beta_1}{\mathcal{X}} \, \phi^*$ also belongs to $L^2(\Omega \setminus D)$.
Then, we also obtain from \eqref{eigen.a.e.2} that 
$$
\int_{\Omega \setminus D} \frac{1 - \mathcal{X}(x) - \beta_1}{\mathcal{X}(x)} \, (\phi^*(x))^2 \, dx 
 =  - \frac 12 \int_{\R^N} \int_{\R^N} J (x-y) (\phi^*(y) - \phi^*(x) )^2 \, dy dx  <  0.  
$$  
Consequently, there exists an open set $\tilde{D} \subset \Omega \setminus D$ such that 
$$
1 - \mathcal{X}(x) - \beta_1  < 0, \quad \forall x \in \tilde{D}.
$$   
Hence, since $0\leq\mathcal{X}(x) \leq 1$, we conclude that $\beta_1 > 0$ finishing the proof. 
\end{proof}

\begin{remark} {\rm
We also observe that $\beta_1 = 1$ is not an eigenvalue of the operator
$$
L^2(\Omega) \ni \phi \quad \to \quad \mathcal{X}(\cdot) \int_{\R^N} J(\cdot-y) \, (\phi(y) - \phi(\cdot) ) \, dy - (1 - \mathcal{X}(\cdot)) \, \phi(\cdot) \in L^2(\Omega).
$$ 
Indeed, it is equivalent to the existence of a nonzero function $\phi^* \in L^2(\Omega)$ satisfying   
\begin{equation} \label{B1eig}
- \phi^*(x) = \int_{\R^N} J(x-y) \, (\phi^*(y) - \phi^*(x) ) \, dy, \quad a.e. \; \Omega \setminus D, 
\end{equation}
with $\phi^*(x) \equiv 0$ in $\R^N \setminus (\Omega \setminus D)$ where 
$D \subset \Omega$ is the set introduced in \eqref{dominioD}. 
Now, it is known from \cite[Proposition 2.2]{ElLibro} that \eqref{B1eig} implies $\phi^*(x) \equiv 0$ in $\Omega \setminus D$.
Since $\phi^*(x) = 0$ in $D$, we get the contradiction.           }
\end{remark}

\begin{remark} {\rm 
Let us suppose that we are in the classic situation in Homogenization Theory in which the family of perforated domains $\Omega^\epsilon$ possesses an
extension operator such as 
$$P_\epsilon \in \mathcal{L}(L^2(\Omega^\epsilon); L^2(\Omega)) \cap \mathcal{L}(H^1(\Omega^\epsilon); H^1(\Omega))$$
with bounded norm (independently of $\epsilon$).
Let us also assume that the data are smooth, by this we mean that $f \in H^1(\Omega)$ and $J \in \mathcal{C}^1(\R^N,\R)$. 
Then, it follows for solutions to the nonlocal Dirichlet problem \eqref{1.1.Dir} and \eqref{Cond.Dir} that $u^\epsilon \in H^1(\Omega^\epsilon)$ and satisfies  
$$
\frac{\partial u^\epsilon}{\partial x_i}(x) = \int_{\R^N}  J(x-y) \Big( \frac{\partial u^\epsilon}{\partial x_i}(y) - 
\frac{\partial u^\epsilon}{\partial x_i}(x) \Big) \, dy - \frac{\partial f}{\partial x_i}(x), 
\quad \textrm{ a.e. } \Omega^\epsilon.
$$ 
Consequently, since $\| \frac{\partial f}{\partial x_i} \|_{L^2(\Omega^\epsilon)}$ is uniformly bounded, from the same arguments used in Lemma \ref{LEBD}, we also get  
$$
\| u^\epsilon \|_{H^1(\Omega^\epsilon)} \leq K
$$ 
for some constant $K>0$ independent of $\epsilon$.
Hence, $P_\epsilon u^\epsilon$ is uniformly bounded in $H^1(\Omega)$, from where we can extract a subsequence (up to a sequence) such that  
\begin{equation} \label{u0}
P_\epsilon u^\epsilon \to u_0 \quad \textrm{ strongly in } L^2(\Omega),
\end{equation}
for some $u_0 \in H^1 (\Omega)$.

Remark that, for nonlocal problems with non-singular kernels it does not hold that the Dirichlet datum is taken continuously, that is, we do not have
$u^\epsilon(x) \to 0$ as $x\to \partial \Omega^\epsilon$ (even for solutions in $C(\Omega^\epsilon)$),
see \cite{Chasseigne}. Therefore, the extension $P_\epsilon u^\epsilon$
of $u^\epsilon$ to the holes does not coincide with the extension by zero, $\tilde u^\epsilon$,
that we have considered here.

Now, we observe that we can pass to the limit (weakly in $L^2$)  in the identity
$$\tilde u^\epsilon(x) = \chi_\epsilon(x) \,  P_\epsilon u^\epsilon(x), \quad \textrm{ a.e. } \Omega.$$
Note that $\tilde u^\epsilon$ is the extension by zero of $u^\epsilon$ in the holes and that 
$\chi_\epsilon \,  P_\epsilon u^\epsilon$ first extends $u^\epsilon$ to the holes and then
multiply by $\chi_\epsilon$.
From the limit of $\tilde u^\epsilon = \chi_\epsilon \,  P_\epsilon u^\epsilon$ we obtain the following relation between the function $u^* \in L^2(\Omega)$ given by Theorem \ref{theoD}, 
and $u_0 \in H^1 (\Omega)$ introduced in \eqref{u0}
$$
u^*(x) = \mathcal{X}(x) \, u_0(x), \qquad \textrm{ a.e. } \Omega.
$$

For the particular case of a periodically perforated domain we have that $ \mathcal{X}(x)$ is a constant
and therefore we obtain that $u^*(x) = \mathcal{X} \, u_0(x)  \in H^1 (\Omega)$. This regularity result
for the limit $u^*$ can also be obtained from the limit problem that is satisfied since we assumed that 
$f \in H^1(\Omega)$ and $J \in \mathcal{C}^1(\R^N,\R)$.
}
\end{remark}

\section{The Neumann problem.} \label{sect-Neumann}
\setcounter{equation}{0}

Given $f \in L^2(\Omega)$, we consider here the following nonlocal problem 
\begin{equation} \label{1.1}
f(x) = \int_{\R^N \setminus A^\epsilon} J (x-y) (u^\epsilon (y) - u^\epsilon (x)) dy, \qquad x  \in \Omega^\epsilon
\end{equation}
with 
\begin{equation} \label{Cond}
u^\epsilon (x) \equiv 0, \qquad x \in \R^N \setminus \Omega.
\end{equation}
Here we are calling $A^\epsilon$ the set
$A^\epsilon = \Omega \setminus \Omega^\epsilon$.

In this problem \eqref{1.1} with \eqref{Cond} we have taken nonlocal Neumann boundary conditions
in the holes and a Dirichlet boundary condition in the exterior of the set $\Omega$. The arguments that we use in this section are similar to the ones for the Dirichlet case. Hence, we concentrate on highlight the main differences.

Since we want to consider general functions $f\in L^2(\Omega)$,
in what follows we need that the first eigenvalue associated to this problem, that is given by
\begin{equation} \label{eigenN.66}
\lambda_1^\epsilon = \inf_{u\in W_\epsilon}  \frac{ \displaystyle
\frac12 \int_{\R^N \setminus A^\epsilon} 
\int_{\R^N\setminus A^\epsilon} J (x-y) (u (y) - u (x))^2 dy\, dx}{
\displaystyle \int_{\Omega^\epsilon} u^2(x) \, dx}
\end{equation}
with
$$
W_\epsilon = \left\{ u\in
L^2({\R^N \setminus A^\epsilon}) \ : \ u (x) \equiv 0, \; x \in \R^N \setminus \Omega \right\},
$$
is strictly positive.

In fact, when $\lambda_1^\epsilon=0$ then our problem 
\eqref{1.1} may not have solutions as the following example shows:
Assume that $\lambda_1^\epsilon=0$ and take $f$ such that 
$\int_{\R^N \setminus A^\epsilon} f \phi_1 \neq 0$ being $\phi_1$ an eigenfunction associated
with $\lambda_1^\epsilon$, that is, $\phi_1$ verifies,
$$0 = \int_{\R^N \setminus A^\epsilon} J(x-y) (\phi_1 (y) - \phi_1 (x)) dy.$$
The existence of such an eigenfunction can be proved as in \cite{jorge}. 
 Then, assume that there is a solution $u^\epsilon$ to the problem 
$$
 f(x) = \int_{\R^N \setminus A^\epsilon} J (x-y) (u^\epsilon (y) - u^\epsilon (x)) dy, \qquad x  \in \Omega^\epsilon.
 $$
 Multiplying by $\phi_1$ and integrating in $\R^N \setminus A^\epsilon$ we get
\begin{equation} \label{LNZ}
\begin{array}{l}
\displaystyle
0\neq \int_{\R^N \setminus A^\epsilon}  f(x) \phi_1(x) = \int_{\R^N \setminus A^\epsilon} \phi_1(x)  \int_{\R^N \setminus A^\epsilon} J (x-y) (u^\epsilon (y) - u^\epsilon (x)) dy\, dx \\[10pt]
\qquad \displaystyle 
= \int_{\R^N \setminus A^\epsilon} u^\epsilon(x)  \int_{\R^N \setminus A^\epsilon} J (x-y) (\phi_1 (y) - \phi_1 (x)) dy\, dx =0,
\end{array}
 \end{equation}
 a contradiction.

A condition under which the first eigenvalue $\lambda_1^\epsilon$ is strictly positive is given in our next result. 
Let us introduce this condition that we call ${\bf (H_N)}$.

 ${\bf (H_N)}$
We assume that there exists a finite family of sets  $B_0, \, B_1, \, ..., \, B_L \subset \R^N \setminus A^\epsilon$
such that $B_0= \R^N \setminus \Omega$,
$$
(\R^N \setminus A^\epsilon) \subset \bigcup_{i=0}^L B_i
\qquad \mbox{and} \qquad
\alpha_{j}=\frac{1}{4}\min_{x\in B_j} \int_{B_{j-1}} J(x-y)\,dy >0.
$$

This condition says that we can cover $\R^N \setminus A^\epsilon$ 
by a finite family of sets $B_i$ starting with $B_0=\R^N \setminus \Omega$
in such a way that every point in some $B_j$ sees the set $B_{j-1}$ with 
uniformly positive probability. This property holds if $A^\epsilon$ is not very thick
(see Example 1 below).

\begin{lemma} \label{lema.posi.neu}
Under the previous condition ${\rm {\bf (H_N)}}$ on the set $A^\epsilon$, 
there exists a positive constant $\lambda$ such that
$$
\int_{{\R^N \setminus A^\epsilon}}
\int_{{\R^N \setminus A^\epsilon}} J(x-y) |u(y)-u(x)|^2\,dy \,dx \geq \lambda \int_{\R^N \setminus A^\epsilon} |u(x)|^2\,dx ,
$$
for every  $u\in W_\epsilon = \left\{ u\in
L^2({\R^N \setminus A^\epsilon}) \ : \ u (x) \equiv 0, \; x \in \R^N \setminus \Omega \right\}$.
\end{lemma}

\begin{proof} Following \cite[Chapter 6]{ElLibro} we cover the domain $\R^N \setminus A^\epsilon$ with a
finite family of disjoint sets, $B_j$  $j=0,1,\cdots, L$ (the existence of such sets is guaranteed by our
hypothesis) and
define
\begin{equation}\label{eq.alphaj}
\alpha_{j}=\frac{1}{4}\min_{x\in B_j} \int_{B_{j-1}} J(x-y)\,dy\,,\qquad
1=\int_{\mathbb R^N} J(s)\,ds\,.
\end{equation}
Now, for $u\in W_\epsilon$, we have
$$
\begin{array}{rl}
\displaystyle
\int_{\R^N \setminus A^\epsilon} \int_{\R^N \setminus A^\epsilon} & J(x-y) |u(y)-u(x)|^2\,dy\,dx\\
\ge & \displaystyle
\int_{B_j} \int_{B_{j-1}} J(x-y) |u(y)-u(x)|^2\,dy\,dx
\end{array}
$$
for $j=1,\cdots,L$, and
$$
\begin{array}{rl}
\displaystyle\int_{B_j} \int_{B_{j-1}} & J(x-y) |u(y)-u(x)|^2\,dy\,dx\\
\ge & \displaystyle \frac{1}{4} \int_{B_j} \left(\int_{B_{j-1}} J(x-y)\,dy\right) |u(x)|^2\,dx  -
\int_{B_{j-1}} \left(\int_{B_j} J(x-y)\,dx\right) |u(y)|^2\,dy\\
\ge & \displaystyle \alpha_j \int_{B_j}|u(x)|^2\,dx - \int_{B_{j-1}} |u(y)|^2\,dy\,.
\end{array}
$$
Then, taking $B_0= \R^N \setminus \Omega$ (notice that $u=0$ in $B_0$), we can iterate this inequality to get that
$$
\int_{B_j} |u(x)|^2\,dx\le C_j \int_{\R^N \setminus A^\epsilon} \int_{\R^N \setminus A^\epsilon} 
J(x-y) |u(y)-u(x)|^2\,dy\,dx\,
$$
where
$$
C_1=\frac{1}{\alpha_1}\,, \qquad
C_{j}=\frac{1}{\alpha_j} (1+ C_{j-1}) \quad j=2,\cdots,L\,.
$$
Therefore, adding in $j$, we have the Poincar\'e type inequality
\begin{equation}\label{eq.poincare}
\int_{\R^N \setminus A^\epsilon}  \int_{\R^N \setminus A^\epsilon}  J(x-y) |u(y)-u(x)|^2\,dy\,dx
\ge \lambda \int_{\R^N \setminus A^\epsilon} |u(x)|^2\,dx, 
\end{equation}
with
$$
\lambda=\left(\sum_{j=1}^L C_j\right)^{-1}\sim \prod_{j=1}^L \alpha_j\,,
$$
as we wanted to show.
\end{proof}

\begin{remark} \label{rem.31} {\rm
Note that if we have that for each $\epsilon$ condition ${\bf (H_N)}$ is satisfied with the number of sets, $L$, independent of $\epsilon$ and a uniform constant $\alpha$ such that $\alpha_j \geq \alpha>0$, then there is a positive constant $c$ independent of $\epsilon$ such that
$$
\lambda_1^\epsilon \geq c >0.
$$
In Section \ref{sect-perodic}, we will see that this property holds, for example, for periodically perforated domains 
in the case that the characteristic function $\mathcal{X} \in L^\infty(\R^N)$ given in \eqref{CharF} is not the null function.
}
\end{remark}

Now, let us observe that existence and uniqueness to our problem
follows considering the variational problem
$$
\min_{u \in W_\epsilon} \frac14 \int_{\R^N \setminus A^\epsilon} \int_{\R^N \setminus A^\epsilon} J (x-y) (u (y) - u (x))^2 dy\, dx -
\int_{\R^N \setminus A^\epsilon} f(x) u (x) \, dx.
$$
In fact, it follows from hypothesis $({\bf H_N})$ (we can use Lemma \ref{lema.posi.neu}) that 
$$
\| u \|^2 := \int_{\R^N \setminus A^\epsilon} \int_{\R^N \setminus A^\epsilon} J (x-y) (u (y) - u (x))^2 dy\, dx
$$
is a norm equivalent to the usual $L^2-$norm in $W_\epsilon$. 
Hence, the functional involved in the minimization problem is lower semicontinuous,
 coercive and convex in $W_\epsilon$, and then possesses a unique minimizer (that we call $u^\epsilon$). 
As for the Dirichlet case, this minimizer is a solution to  \eqref{1.1} and \eqref{Cond}.
Moreover, if $u^\epsilon \in L^2(\Omega^\epsilon)$ satisfies \eqref{1.1} and \eqref{Cond}, then it is 
a minimizer.
Multiplying the equation by $u^\epsilon$ we get
\begin{equation} \label{1.1.weak.22}
\begin{array}{l}
\displaystyle - \int_{\Omega^\epsilon}  f(x) u^\epsilon (x) \, dx  \displaystyle 
=  \frac12
  \int_{\R^N \setminus A^\epsilon}  \int_{\R^N \setminus A^\epsilon} J (x-y) (u^\epsilon (y) - u^\epsilon (x))^2\, dy dx 
  \geq \lambda_1^\epsilon  \int_{\Omega^\epsilon} (u^\epsilon(x))^2 \, dx.
  \end{array}
\end{equation}
Here $\lambda_1^\epsilon$ is the first eigenvalue \eqref{eigenN.66} associated with this operator in the space
$W_\epsilon$. 
Therefore, assuming that $\lambda_1^\epsilon \geq c>0$ with $c$ independent of $\epsilon$ (see Remark \ref{rem.31}), we get that
$$
\int_{\Omega^\epsilon} (u^\epsilon(x))^2 \, dx
$$
is bounded by a constant that depends only on $f$ but is independent of $\epsilon$. Hence, along a subsequence if necessary, 
$$
\tilde u^\epsilon \rightharpoonup u^*
$$
weakly in $L^2 (\Omega)$ as $\epsilon \to 0$ where $\tilde \cdot$ denotes the extension by zero on functions defined in the open set $\Omega^\epsilon$.
Hence, if $\chi_\epsilon$ is the characteristic function of $\Omega^\epsilon$ and $\tilde u^\epsilon$ is the extension by zero of  
$u^\epsilon$ to $\R^N$, we can write the weak form of the equation as
\begin{equation} \label{1.2}
\begin{array}{l}
 \displaystyle\int_{\Omega}  \chi_\epsilon(x) \, \varphi (x) f(x) \, dx  =  \int_{\Omega}  \chi_\epsilon(x) \, \varphi (x)  
 \left( \int_{\R^N} J(x-y) \, \tilde u^\epsilon (y) \, dy \right) dx \displaystyle 
\\[10pt]
  \qquad \qquad \displaystyle 
  - 
  \int_{\Omega}  \varphi (x) \, \tilde u^\epsilon(x)  \int_{\R^N \setminus A^\epsilon} J(x-y) \, dy \, dx
  \end{array}
\end{equation}
for any $\varphi \in L^2(\Omega)$.

%
%
%

Now let us give the proof of Theorem \ref{theo1}.

\begin{proof}[Proof of Theorem \ref{theo1}]
Under the assumption $\lambda_1^\epsilon \geq c>0$ with $c$ independent of $\epsilon$, 
we have already seen that 
$$
\| u^\epsilon \|_{L^2(\Omega^\epsilon)} \leq c^{-1} \| f \|_{L^2(\Omega)}
$$from where we obtain (up to a subsequence) the weak convergence 
$$
\tilde u^\epsilon \rightharpoonup u^*.
$$

First, we observe that the additional assumption $\mathcal{X}(x) = 0$ implies $u^*(x) = 0$ in $\R^N$.
In fact, it is a consequence of the lower semicontinuity of the norm and the limit  
$$
\left| \int_\Omega \chi_\epsilon(x) \, \varphi(x) \, u^\epsilon(x) \, dx \right| \leq \| \chi_\epsilon \varphi \|_{L^2(\Omega^\epsilon)} \| u^\epsilon \|_{L^2(\Omega^\epsilon)} 
\to 0, \qquad \mbox{as $\epsilon \to 0$ } \quad \forall \varphi \in L^2(\Omega),
$$

Now let us get the limit problem to the other cases. 
In order to do that, we need to pass to the limit in \eqref{1.2}.
Let us introduce the function
$$
O^\epsilon(x) = \int_{\R^N \setminus A^\epsilon} J(x-y) \, dy.
$$
We want to show that 
\begin{equation} \label{OL2C}
O^\epsilon \to O_0 \quad \textrm{ strongly in } L^2(\Omega)
\end{equation}
with
$$
O_0(x) = \int_{\R^N} J(x-y) \left( 1 - \chi_\Omega (y) + \mathcal{X}(y)  \right) dy.
$$

Since
\begin{eqnarray*}
O^\epsilon(x) & = & \int_{\R^N \setminus A^\epsilon} J(x-y) \, dy \\
& = & \int_{\R^N} J(x-y) \, dy - \int_{\Omega} J(x-y) \, dy + \int_{\Omega^\epsilon} J(x-y) \, dy
\end{eqnarray*}
for any $x \in \R^N$, we just need to pass to the limit in  
$$
\hat O^\epsilon(x) = \int_{\Omega^\epsilon} J(x-y) \, dy.
$$
Now, observe that 
$$
\hat O^\epsilon(x) = \int_{\Omega^\epsilon} J(x-y) \, dy = \int_{\Omega} \chi_\epsilon(y) \, J(x-y) \, dy,
$$
and then, from \eqref{CharF} we get 
$$
\hat O^\epsilon(x) \to \hat O_0(x) = \int_\Omega \mathcal{X}(y) \, J(x-y) \, dy 
$$
for all $x \in \R^N$.
Thus, since 
$$
|O^\epsilon(x)| \leq 1 \quad \forall \epsilon>0 \textrm{ and } x \in \R^N,
$$
we can argue as in \eqref{UL2C} to obtain from Dominated Convergence Theorem that
$$
\hat O^\epsilon \to \hat O_0 \quad \textrm{ strongly in } L^2(\Omega).
$$
Consequently we conclude \eqref{OL2C}.

Now we can combine \eqref{CharF}, \eqref{UL2C} and \eqref{OL2C} to pass to the limit in \eqref{1.2} obtaining 
$$
\begin{array}{l}
 \displaystyle\int_{\Omega}  \mathcal{X}(x) \, \varphi (x) f(x) \, dx  =  \int_{\Omega}  \mathcal{X}(x) \, \varphi (x)  
 \left( \int_{\R^N} J(x-y) \, u^*(y) \, dy \right) dx \displaystyle 
\\[10pt]
  \qquad \qquad \displaystyle 
  - 
  \int_{\Omega}  \varphi (x) \, u^*(x)  \int_{\R^N} J(x-y) \left( 1 - \chi_\Omega (y) + \mathcal{X}(y)  \right) dy \, dx
  \end{array}
$$
for any $\varphi \in L^2(\Omega)$. That can be rewriten as 
$$
\begin{array}{l}
 \displaystyle\int_{\Omega}  \mathcal{X}(x) \, \varphi (x) f(x) \, dx  =  \int_{\Omega}  \mathcal{X}(x) \, \varphi (x)  
  \int_{\R^N} J(x-y) \,( u^*(y) - u^*(x) ) \, dy \, dx \displaystyle 
\\[10pt]
  \qquad \qquad \displaystyle 
  +
  \int_{\Omega}  \varphi (x) \, u^*(x)  \left(  \mathcal{X}(x) - \int_{\R^N} J(x-y) \left( 1 - \chi_\Omega (y) + \mathcal{X}(y)  \right) dy  \right) dx
  \end{array}
$$
by the addition and subtraction of the term 
$$
\int_\Omega \varphi(x) \, \mathcal{X}(x) \, u^*(x) \, dx = \int_\Omega \varphi(x) \, \mathcal{X}(x) \int_{\R^N} J(x-y) \, u^*(x) \, dy \, dx
$$
since $\int_{\R^N} J(x-y) \, dy =1$.
Observe that $u^*(x)  \equiv 0$ wherever $x \in \R^N \setminus \Omega$.

Therefore, there exists $u^* \in L^2(\Omega)$ with $\tilde u^\epsilon \rightharpoonup u^*$ weakly in $L^2(\Omega)$ and $u^*$ satisfying the following nonlocal problem 
\begin{equation} \label{1.1.b}
\mathcal{X}(x) \, f(x) = \mathcal{X}(x) \int_{\R^N}  J (x-y) (u^*(y) - u^*(x) ) \, dy -\Lambda(x) \, u^*(x) 
\end{equation}
with 
$u^*(x) \equiv 0$, $x \in \R^N \setminus \Omega$,
where $\Lambda \in L^\infty(\Omega)$ is given by 
\begin{equation} \label{Lambda}
\Lambda(x) =  \int_{\R^N} J(x-y) \, ( 1- \chi_\Omega (y) + \mathcal{X}(y) ) \, dy - \mathcal{X}(x), \quad x \in \Omega,
\end{equation}
completing the proof of Theorem \ref{theo1}.

Uniqueness for the limit problem can be obtained as in the proof of Theorem \ref{theoD}.
This fact implies the convergence of whole sequence $\tilde u^\epsilon$.
\end{proof}

\begin{remark} {\rm
Observe that we still can reach the same result described in Theorem \ref{theo1} for the more general situation 
$$
f^\epsilon(x) = \int_{\R^N \setminus A^\epsilon} J(x-y) (u^\epsilon(y) - u^\epsilon(x)) \, dy
$$ 
with $f^\epsilon \to f$ strongly in $L^2(\Omega)$ since 
$\|f^\epsilon\|_{L^2(\Omega)}$ remains uniformly bounded in $\epsilon$ and 
$$
\int_\Omega \chi_\epsilon(x) \, f^\epsilon(x) \, dx \to \int_\Omega \mathcal{X}(x) \, f(x) \, dx
$$
as $\epsilon \to 0$.}
\end{remark}

We still notice that a stronger hypothesis is needed to obtain strong convergence in $L^2$ of the extended solutions to zero for the singular case $\mathcal{X}(x) = 0$ a.e. $\R^N$. 
In this respect, we have the following result.

\begin{cor}   \label{CNP} 
Let $\Gamma: \Omega \mapsto \R$ be the function given by
\begin{equation} \label{Gamma}
\Gamma(x) = \int_{\R^N \setminus \Omega} J(x-y) \, dy. 
\end{equation}
Suppose $\mathcal{X}(x)=0$ a.e. $\Omega$ in \eqref{CharF} and  assume
\begin{equation} \label{HGamma}
\Gamma(x) \geq m >0, \quad \textrm{ in } \Omega
\end{equation}
for some $m>0$. 

Then, the solutions $u^\epsilon$ of the Neumann problem \eqref{1.1} with \eqref{Cond} satisfies  
$$\| u^\epsilon \|_{L^2(\Omega^\epsilon)} \to 0, \quad \textrm{ as } \epsilon \to 0.$$
\end{cor}
\begin{proof}
First, let us take $\varphi = \chi_\epsilon \, u^\epsilon$ in \eqref{1.2}. Then,  
\begin{equation} \label{1.2Cor}
 \displaystyle\int_{\Omega}  \tilde u^\epsilon (x) f(x) \, dx  =  \int_{\Omega}  \tilde u^\epsilon (x)  
 \left( \int_{\R^N} J(x-y) \, \tilde u^\epsilon (y) \, dy \right) dx  
  - 
  \int_{\Omega}  {\tilde u^\epsilon(x)}^2  \int_{\R^N \setminus A^\epsilon} J(x-y) \, dy \, dx.
\end{equation}

Since $\Gamma$ is strictly positive in $\Omega$, hypothesis (${\bf H_N}$) holds, and then 
$$
\tilde u^\epsilon \rightharpoonup 0, \quad \textrm{ weakly in } L^2(\Omega), 
$$
as $\epsilon \to 0$. 
Consequently, due to \eqref{1.2Cor}, we obtain 
\begin{equation} \label{SLM100}
\lim_{\epsilon \to 0}  \int_{\Omega}  {\tilde u^\epsilon(x)}^2  \int_{\R^N \setminus A^\epsilon} J(x-y) \, dy \, dx = 0.
\end{equation}
On the other hand, we have that 
\begin{equation} \label{SLM101}
\left| \int_{\Omega}  {\tilde u^\epsilon(x)}^2  \int_{\R^N \setminus A^\epsilon} J(x-y) \, dy \, dx \right| 
\geq \| u^\epsilon \|_{L^2(\Omega^\epsilon)} \inf_{x \in \bar \Omega} \int_{\R^N \setminus A^\epsilon} J(x-y) \, dy. 
\end{equation}

Hence, we conclude the proof from \eqref{HGamma}, \eqref{SLM100} and \eqref{SLM101}.
\end{proof}

Finally, we observe that in this section, we have not proceeded as in the proof of Lemma \ref{LEBD} to obtain a positive lower bound to the eigenvalues. 
In fact, we can not pass to the limit in the eigenvalue problem associated to the Neumann equation \eqref{1.1} with \eqref{Cond}
$$
- \lambda_1^\epsilon \int_{\Omega^\epsilon} \phi^\epsilon(x) \, \varphi(x) \, dx 
= \int_{\Omega^\epsilon} \varphi(x) \int_{\R^N \setminus A^\epsilon} J(x-y) ( \phi^\epsilon(y) - \phi^\epsilon(x) ) \, dy dx, \quad x \in \Omega^\epsilon,
$$
with $\phi^\epsilon$ and $\varphi \in L^2(\R^N)$ vanishing in $\R^N \setminus \Omega$.
We can extract a convergent subsequence of eigenvalues and eigenfunctions, but we are not able to evaluate their limits in order to show that they are non trivial.
This drives us to hypothesis $({\bf H_N})$ and the following example that shows that in fact there are
configurations of the holes for which the first eigenvalue is zero.

\begin{figure}[htp] 
\centering \scalebox{0.4}{\includegraphics{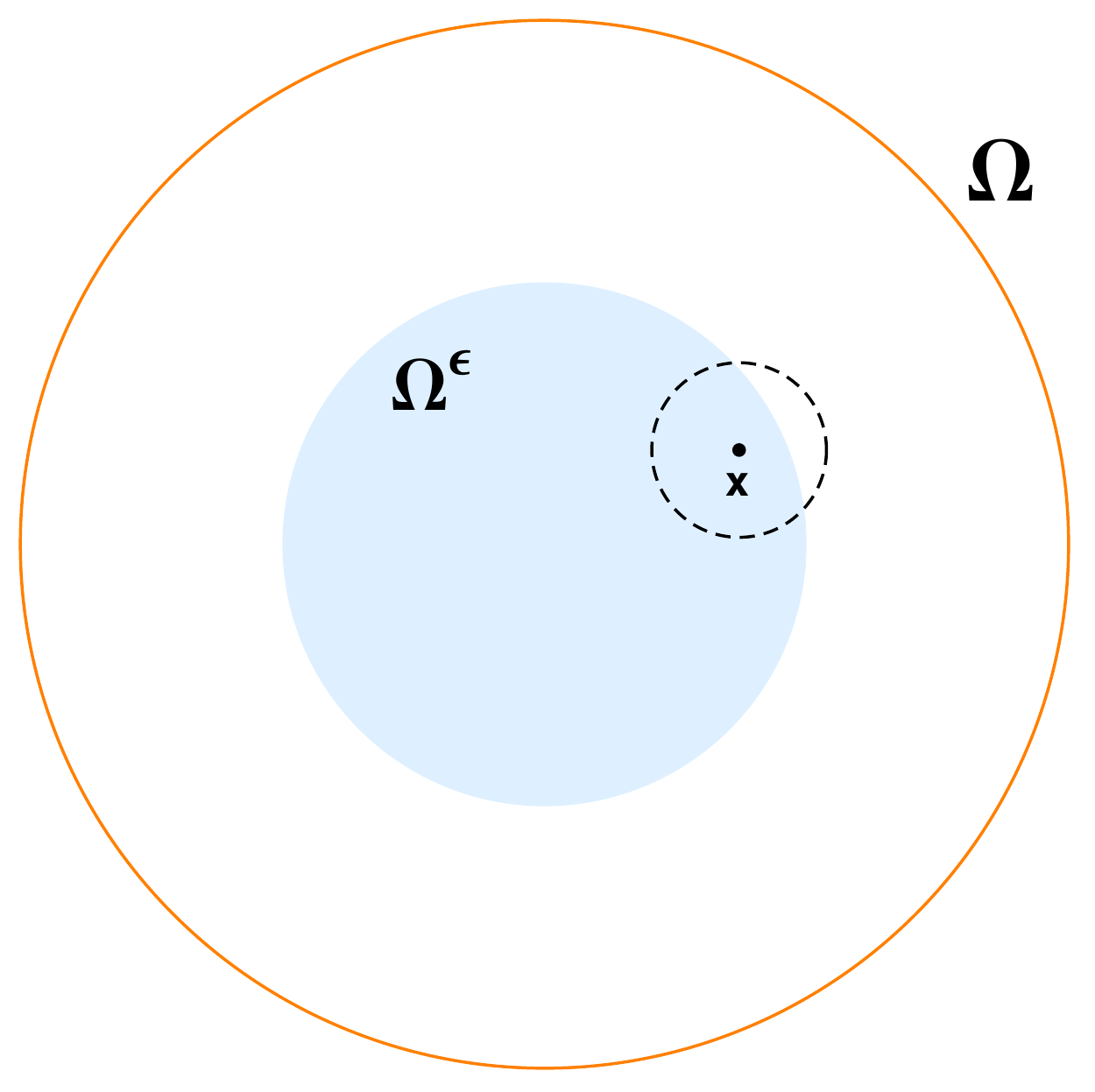}}
\caption{An example where the first eigenvalue to the Neumann problem is zero.}
\label{fig2} 
\end{figure}

\noindent{\bf Example 1.}
As illustrated in Figure \ref{fig2}, let $\Omega = B_{6}$ and $\Omega^\epsilon = B_{3}$ the balls of radius $6$ and $3$ centered at the origin in $\R^N$ respectively, 
hence the annulus $A^\epsilon = \Omega \setminus \Omega^\epsilon$ is the hole.
Now assume that the kernel $J$ satisfies hypothesis $({\bf H_J})$ and has the unit ball $B_1$ as its support.
Then, we have that the function 
$$
u(x) = \left\{
\begin{array}{ll}
1, & x \in \Omega^\epsilon \\
0, & x \in \R^N \setminus \Omega
\end{array}
\right.
$$
satisfies 
$$
\frac{ \displaystyle
\frac12 \int_{\R^N \setminus A^\epsilon} 
\int_{\R^N\setminus A^\epsilon} J (x-y) (u(y) - u (x))^2 dy\, dx}{
\displaystyle \int_{\Omega^\epsilon} u^2(x) \, dx} = 0
$$
since ${\rm supp} (J) = B_1$ implies $J(x-y) (u(y) - u (x)) = 0$ whenever $x$, $y \in \R^N \setminus A^\epsilon$.
Consequently, hypothesis $({\bf H_J})$ it is not enough to guarantee the first eigenvalue of the Neumann problem introduced in \eqref{eigenN.66} 
is strictly positive as in the Dirichlet case.

\section{Periodically Perforated Domains} \label{sect-perodic}
\setcounter{equation}{0}

In this section we discuss Theorems \ref{theoD} and \ref{theo1} in the particular situation where the holes in $\Omega \subset \R^N$ are periodically distributed. 

First we deal with the critical case, it means, that one in which the size and distribution of holes possess the same order.
Next we consider the other cases.

\subsection{Size and distribution with same order} \label{SSPPD1} 

Let $Q \subset \R^N$ be the \emph{representative cell} 
$$
Q = (0,l_1) \times (0,l_2) \times ... \times (0,l_N).
$$ 
We perforate $\Omega$ removing from it a set $A^\epsilon$ of periodically distributed holes given as follows:
Take any open set $A \subset Q$ such that $T=Q \setminus A$ is measurable set satisfying $|T| \neq 0$.
Denote by $\tau_\epsilon(A)$ the set of all translated images of $\epsilon \bar A$ of the form 
$\epsilon(kl + A)$ where $k \in \mathbb{Z}^N$ and $kl=(k_1 l_1, ..., k_N l_N)$. 
Now define
$$
A^\epsilon = \Omega \cap \tau_\epsilon(A).
$$

We introduce our perforated domain as 
\begin{equation} \label{PerfDom}
\Omega^\epsilon = \Omega \setminus A^\epsilon.
\end{equation}

Note that when considering $\Omega^\epsilon$ we have removed from $\Omega$ a large number of holes
of size $|\epsilon \bar A|$ which are $\epsilon$-periodically distributed. 
$A^\epsilon$ represent the sets of holes inside $\Omega$. 
It contains the interior holes, that ones that are fully contained into $\Omega$, as well as part of 
each hole that intersects the boundary $\partial \Omega$.

We now pass to the limit in the characteristic function $\chi_\epsilon$ to obtain the limit equations to Dirichlet and Neumann problems in the
family of perforated domains given by \eqref{PerfDom}.
To do that let $\chi_A$ be the characteristic function of the open set $A \subset Q$ extended periodically in $\R^N$. 
Hence, if $\chi_{A^\epsilon}$ is the characteristic function of $A^\epsilon$, for each $x \in A^\epsilon$
there exist $k \in \mathbb{Z}^N$ such that  
$$
\chi_{A^\epsilon}(x) = \chi_A\left(\frac{x-\epsilon k l}{\epsilon}\right) = \chi_A(x/\epsilon).
$$

Therefore, if $\chi_\Omega$ and $\chi_\epsilon$ are the characteristic functions of $\Omega$ and $\Omega^\epsilon$ respectively, 
we have the following relationship  
\begin{equation} \label{eqCF12}
\chi_\epsilon(x) = \chi_\Omega(x) - \chi_{A^\epsilon}(x). 
\end{equation}
It follows from the Average Theorem \cite[Theorem 2.6]{CD} that  
\begin{equation} \label{eqCF1}
\chi_{A^\epsilon} \rightharpoonup \frac{1}{|Q|} \int_{Q} \chi_A(s) \, ds = \frac{|A|}{|Q|}, \quad \textrm{ as } \epsilon \to 0, 
\end{equation}
weakly star in $L^\infty(\Omega)$. Hence, from \eqref{eqCF12}, \eqref{eqCF1} and 
$|T| = |Q\setminus A| = |Q| - |A|$ we obtain that 
$$
\chi_{\epsilon} \rightharpoonup \frac{|T|}{|Q|} \quad \textrm{ weakly$^*$ in } L^\infty(\Omega).
$$

In this way, we can set 
\begin{equation} \label{PDX}
\mathcal{X}(x) = \frac{|T|}{|Q|} \chi_\Omega(x) \quad \textrm{ in } \R^N
\end{equation}
at assumption \eqref{CharF}.

Thus, as a consequence of Theorem \ref{theoD} and \eqref{PDX}, the extended solutions $\tilde u^\epsilon$ of the Dirichlet problem \eqref{1.1.Dir} and \eqref{Cond.Dir} weakly converge in $L^2(\Omega)$ to the solution $u^*$ to 
\begin{equation} \label{VLPPDD}
\begin{array}{l}
 \displaystyle \frac{|T|}{|Q|} f(x) = \frac{|T|}{|Q|} \int_{\R^N} J(x-y) ( u^*(y) - u^*(x) ) \, dy  
- \left( \frac{|Q| - |T|}{|Q|} \right) \, u^*(x), \qquad x \in \Omega,
  \end{array}
\end{equation}
with
$$
u^* (x) \equiv 0, \qquad x \in \R^N \setminus \Omega,
$$
which can be rewritten as 
\begin{equation}  \label{PPDD}
\begin{array}{l}
 \displaystyle  f(x) =  \int_{\R^N} J(x-y) ( u^*(y) - u^*(x) ) \, dy 
 - \left( \frac{|Q| - |T|}{|T|} \right) \, u^*(x), \qquad x \in \Omega,
  \end{array}
\end{equation}
whenever $|T| \neq 0$.

Now to get the limit equation to the Neumann problem \eqref{1.1} and \eqref{Cond},
we first have to see that condition $({\bf H_N})$ is verified. 
To do that, given $\delta>0$ small we define the sets
$$
B_0 = \R^N \setminus \Omega
$$
and
$$
B_j =\{x\in \Omega^\epsilon \,:\,  \delta j <d(x,\partial \Omega)<\delta (j+1) \}\, \qquad j=1,..., L.
$$
Note that 
$$
(\R^N \setminus A^\epsilon) \subset \bigcup_{j=0}^L B_j
$$
Notice also that, for $x\in B_j$ we have 
\begin{equation} \label{HNV}
\begin{array}{l}
\displaystyle
\int_{B_{j-1}} J(x-y)\,dy \geq \int_{\{ y \in \Omega \setminus A^\epsilon \, : \, \delta (j-1) + \delta/2 <d(y,\partial \Omega) < \delta j \}}  J(x-y)\,dy \\[10pt]
\displaystyle
\qquad \geq \left(\min_{|z| \leq 1-\delta/2} J (z) \right) \times \Big| \{ y \in \Omega \setminus A^\epsilon \, : \, \delta (j-1) + \delta/2 <d(y,\partial \Omega) < \delta j \} \Big| >0.
\end{array}
\end{equation}
Here we have used that $J$ is continuous with $J(0) >0$.

Also observe that the number of sets, $L$, as well as the lower bound for the $\alpha_j$, depend only on $\delta$ 
and therefore they can be chosen independently of $\epsilon$.

Therefore, we can conclude from Lemma \ref{lema.posi.neu} that there exists $\epsilon_0$ and $c>0$ such that 
$$
\lambda^\epsilon_1 > c > 0, \qquad \forall \epsilon \in (0, \epsilon_0),
$$
and then, due to Theorem \ref{theo1}, we obtain that the limit equation of \eqref{1.1} and \eqref{Cond} is   
\begin{equation} \label{VLP1}
\begin{array}{l}
 \displaystyle  f(x) =  \int_{\R^N} J(x-y) ( u^*(y) - u^*(x) ) \, dy 
 \\[10pt]
  \qquad \qquad \qquad \qquad \displaystyle 
- \left( \frac{|Q| - |T|}{|T|} \right) \, u^*(x) \, \int_{\R^N \setminus \Omega} J(x-y) \, dy, \qquad x \in \Omega,
  \end{array}
\end{equation}
with $u^* (x) \equiv 0$ for all $x \in \R^N \setminus \Omega$, 
where the term $\Lambda$ defined in \eqref{Lambda} can be calculated by  
\begin{eqnarray*}
\Lambda(x) & = & \int_{\R^N} J(x-y) \left( 1 - \chi_\Omega(y) + \frac{|T|}{|Q|} \chi_\Omega(y) \right) dy - \frac{|T|}{|Q|} \chi_\Omega(x) \\
& = & \left(\frac{|Q|-|T|}{|Q|} \right) \int_{\R^N} J(x-y) \left( 1 - \chi_\Omega(y) \right) dy. 
\end{eqnarray*}

\begin{remark} 
{\rm 
Notice that what makes distinction between the limit problems \eqref{VLPPDD} and \eqref{VLP1} here is just the coefficient \eqref{Gamma}.
 }
\end{remark}

\begin{figure}[htp] 
\centering \scalebox{0.4}{\includegraphics{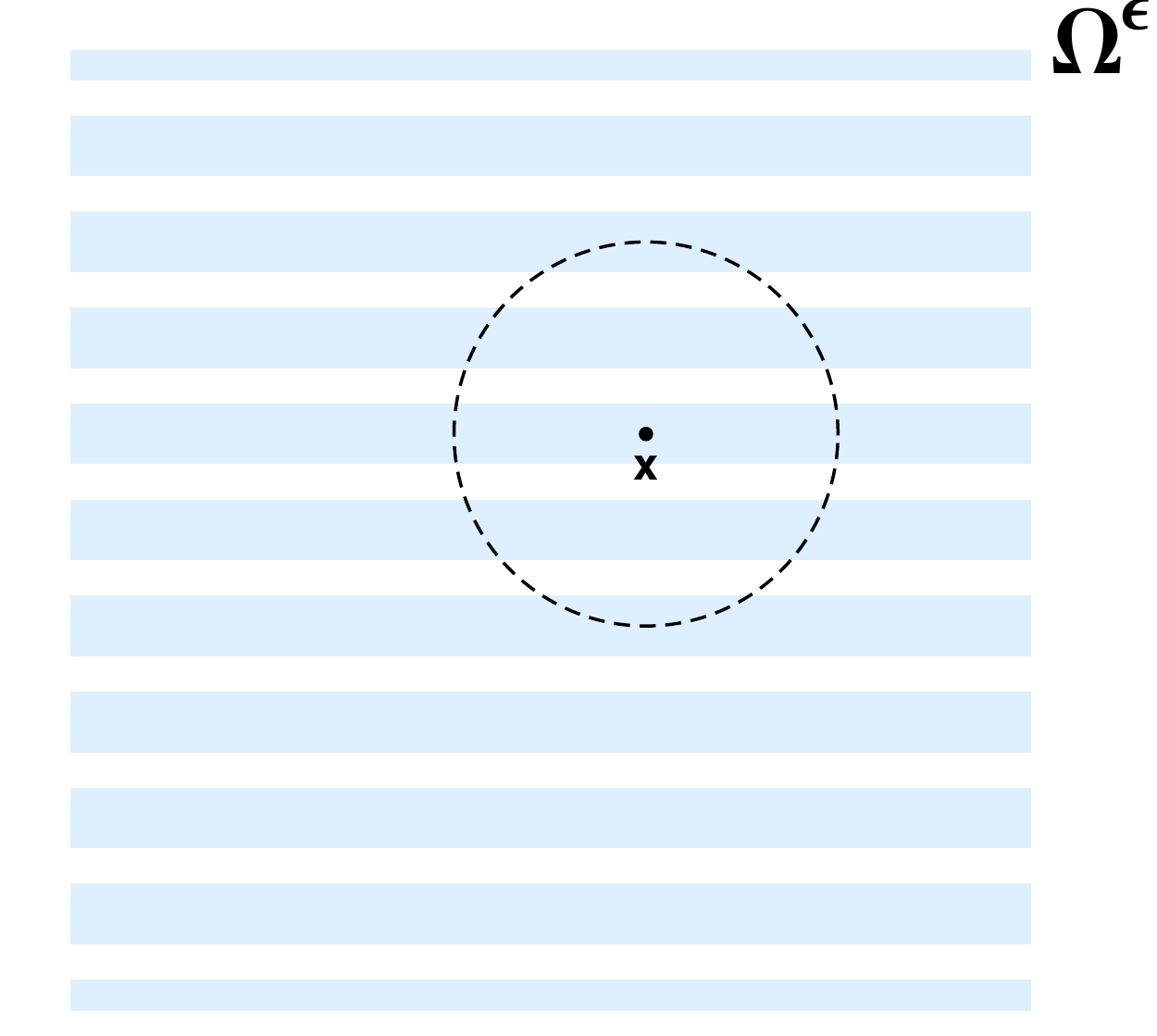}}
\caption{A non-connected perforated domain.}
\label{fig3} 
\end{figure}

Now, let us present an example that shows that we do not need the set $\Omega^\epsilon$
to be connected. This has to be contrasted with what happens in the local case (where connectedness
of $\Omega^\epsilon$ plays a crucial role).

\noindent{\bf Example 2.} We consider $\Omega = (-1,1) \times (-1,1)$ the periodic case in which the hole is given by
the scaled version of the strip
$$
A = [0,1] \times [1/3,2/3] ,
$$
in the representative cell $(0,1) \times (0,1)$. That is, we remove from $\Omega$
the set of all translated images of $\epsilon A$ of the form 
$\epsilon(k + A)$ where $k \in \mathbb{Z}^N$. 
Note that $\Omega^\epsilon$ is a square from where we have removed a large number of horizontal strips
of size $|\epsilon A|$ which are $\epsilon$-periodically distributed. 
See Figure \ref{fig3} which illustrates this situation.
Considering $B_1,....,B_L$ as horizontal strips of width $\delta$
one can easily check that hypothesis $({\bf H_N})$ holds (note that here $\delta$ and $L$ can be 
chosen independently of $\epsilon$).

\subsection{Size and distribution with different orders}

Now let us take holes in the open set $\Omega$ with distribution and size of different orders.
To do that we consider the sets $Q$, $A$ and $B$ as in Section \ref{SSPPD1}, 
and define by $\hat \tau_\epsilon(A)$ the union of all translated images of $\epsilon^\gamma \bar A$ of the form 
$$\epsilon kl + \epsilon^{\gamma}A = \epsilon(kl + \epsilon^{\gamma-1}A)$$ 
where $\gamma>0$, $k \in \mathbb{Z}^N$ and $kl=(k_1 l_1, ..., k_N l_N)$. 
Here the hole set are given by 
$$
A^\epsilon = \Omega \cap \hat \tau_\epsilon(A)
$$
and the perforated domain by 
$$
\Omega^\epsilon = \Omega \setminus A^\epsilon.
$$

When $\gamma<1$, the order of distributions are bigger that the size of the holes at $\epsilon=0$. 
On the other hand, the size order of the holes are larger than their distributions as $\gamma > 1$.
The case $\gamma=1$ has beed considered in Section \ref{SSPPD1} and the sets are of the same order of the size and distribution of a typical cell.

If we assume $\gamma>1$, the set of distributed holes vanishes in $\Omega$ as $\epsilon \to 0$, and then, $\Omega^\epsilon$ fills the whole of $\Omega$ which implies  
\begin{equation} \label{CharFOmega}
\chi_{\epsilon} \rightharpoonup \chi_\Omega \quad \textrm{ weakly$^*$ in } L^\infty(\Omega)
\end{equation}
where $\chi_\Omega$ is the characteristic function of $\Omega$.

Then, as a consequence of Theorems \ref{theoD} and \ref{theo1}, 
we get the same limit equation to both Dirichlet and Neumann problems \eqref{1.1.Dir} and \eqref{1.1} respectively.
We get the nonlocal Dirichlet problem in the open set $\Omega$
\begin{equation} 
\begin{array}{l}
 \displaystyle  f(x) = \int_{\R^N} J(x-y) ( u^*(y) - u^*(x) ) \, dy, \quad x \in \Omega, 
 \end{array}
\end{equation}
with
$u^* (x) \equiv 0$, in $\R^N \setminus \Omega$.

Notice that hypothesis (${\bf H_N}$) to the Neumann problem can be easy verified here as in \eqref{HNV}. 
Among other things, it means that under the assumptions \eqref{CharFOmega}, (${\bf H_J}$) and (${\bf H_N}$), the nonlocal Neumann problem \eqref{1.1} behaves as the non local Dirichlet problem for $\epsilon$ small enough.

Finally, if $\gamma <1$, it is not difficult to see that the set of the holes fill the whole of $\Omega$ when $\epsilon$ goes to zero implying that 
$$
\chi_{\epsilon} \rightharpoonup 0 \quad \textrm{ weakly$^*$ in } L^\infty(\Omega).
$$ 
Hence, it follows from Theorems \ref{theoD} and \ref{theo1} that the extended solutions $\tilde u^\epsilon$ 
to both Neumann and Dirichlet problems weakly converge to zero in $L^2(\Omega)$ as $\epsilon \to 0$. 
Moreover, from Corollary \ref{corD}, we still obtain strong convergence to zero in the Dirichlet case. 
That is, if $u^\epsilon$ satisfies the Dirichlet problem \eqref{1.1.Dir} and \eqref{Cond.Dir}, then 
$$
\tilde u^\epsilon \to 0 \quad \textrm{ strongly in } L^2(\Omega)
$$ 
as $\epsilon \to 0$.

With the additional condition \eqref{HGamma}, we also have $\tilde u^\epsilon \to 0$ strongly in $L^2(\Omega)$
for the Neumann problem \eqref{1.1} with \eqref{Cond}. 
See Corollary \ref{CNP}.

\section{Rescaling the kernel} \label{sect-rescales}
\setcounter{equation}{0}

In this section our aim is to see how these problems behave when we rescale the kernel in order
to approximate local equations. Hence, now we have two different parameters $\epsilon$ (that controls
the size of the holes) and $\delta$ (that is used to rescale the kernel). Our aim is to study the limits $\epsilon \to 0$
(to have an homogenized limit problem) and $\delta \to 0$ (to approach local problems).

To describe accurately these limits along this section we need to restrict ourselves to the periodic case. We assume that
we have a bounded domain $\Omega$
from where we have removed a big number of periodic small balls (the holes). 
That is, we consider $\Omega^\epsilon = \Omega \setminus \cup B_{r^\epsilon} (x_i)$
where $B_{r^\epsilon} (x_i)$ is a ball centered 
in $x_i\in \Omega$ of the form $x_i \in 2 \epsilon \Z^N$ with $0 < r^\epsilon < \epsilon \leq 1$. 
To simplify, we only remove here balls that are strictly contained in $\Omega$.
Moreover, we will assume here that $J$ is radially symmetric.

Our aim is to show that, in general, the involved limits do not commute, but it holds that both limits
exist. For $u^{\epsilon, \delta}$ a solution to the nonlocal problem in $\Omega^\epsilon$ with
a kernel rescaled with $\delta$ we have that
$$
\lim_{\delta \to 0} \lim_{\epsilon \to 0} \tilde{u}^{\epsilon, \delta} = w,
$$
and
$$
 \lim_{\epsilon \to 0} \lim_{\delta \to 0} \tilde{u}^{\epsilon, \delta} = v,
$$
but, in general 
$$
w\neq v.
$$

\subsection{The Dirichlet case}

We start by considering
\begin{equation} \label{Dir-ep-delsec}
f(x) =\frac{C}{\delta^{N+2}} \int_{\R^N} J \left(\frac{x-y}{\delta}\right) 
(u^{\epsilon,\delta} (y) - u^{\epsilon,\delta} (x)) dy, \qquad x \in \Omega^\epsilon,
\end{equation}
with 
\begin{equation} \label{Cond.Dir.sec}
u^{\epsilon,\delta} (x) \equiv 0, \qquad x \in \R^N \setminus \Omega^\epsilon ,
\end{equation}
and $C$ is a normalizing constant given by 
\begin{equation} \label{constC.Dir}
C = \left( \frac{1}{2} \int_{\R^N} J(x) \, x_{1}^2 \, dx \right)^{-1}
\end{equation}
where $x_{1}$ is the first coordinate of $x \in \R^N$.

Existence and uniqueness of the solutions $u^{\epsilon,\delta}$ of \eqref{Dir-ep-delsec} are guaranteed in Section \ref{sect-Dirichlet} for any $f \in L^2(\Omega)$, $\epsilon$ and $\delta>0$.
Thus, we proceed with the analysis of the behavior of $u^{\epsilon,\delta}$ as $\epsilon$ and $\delta$ go to zero.  

Performing the change of variable $$z=\frac{(x-y)}{\delta} $$ and using Taylor expansion, we obtain 
for a smooth $u$
\begin{equation} \label{NPRN.jj}
\begin{array}{l}
\displaystyle  C \frac{1}{\delta^{2+N}} \int_{\R^N} J\left(\frac{x-y}{\delta}\right) \left( u(y) - u(x) \right) dy =  \frac{C}{\delta^2} \int_{\R^N} J(z) \left( u(x-\delta z) - u(x) \right) dz  \\[10pt]
\qquad \displaystyle =  \frac{C}{2} \left[ \sum_{i=1}^{N} \partial_i^2 u(x)  \right] \int_{\R^N} J(z) z^2_{1} dz + O(\delta)  =  \Delta_{x} u(x) + O(\delta) .
\end{array}
\end{equation}

Expression \eqref{NPRN.jj} makes a connection between the non-local problem with the kernel $J_{\delta}$ 
and the following boundary value problem   
\begin{equation} \label{LNP-jj}
\left\{
\begin{array}{ll}
\displaystyle \Delta v^\epsilon(x)  = f(x), \qquad & x \in \Omega^\epsilon \\[6pt]
v^\epsilon (x) = 0, 
\qquad & x \in \partial \Omega^\epsilon.
\end{array}
\right. 
\end{equation}

Now,   
we can use the results in \cite[Section 3.2.2]{ElLibro} (see also \cite{CER1}) to obtain 
that
\begin{equation} \label{LD}
\| u^{\delta,\epsilon} - v^\epsilon \|_{L^2(\Omega)} \to 0, \quad \textrm{ as } \delta \to 0.
\end{equation}

As we can see from \cite{CM}, problem \eqref{LNP-jj} is the prototype for the study of the Laplacian in perforated domains.  As we have mentioned in the introduction, the limit as $\epsilon \to 0$ is given by 
\eqref{v^*.intro}.
Hence, we have that 
$$
\| v^\epsilon - v \|_{L^2(\Omega)} \to 0, \quad \textrm{ as } \epsilon \to 0.
$$

Now, we just observe that we have obtained the following result.

\begin{theorem} \label{teo.rescalesDir.1}
If $u^{\delta,\epsilon} $, $v$ are given by \eqref{Dir-ep-delsec} and \eqref{v^*.intro} respectively, it holds that 
$$
\lim_{\epsilon \to 0} \left( \lim_{\delta \to 0} u^{\delta,\epsilon} \right) = v \quad \textrm{ in } L^2(\Omega),
$$
\end{theorem}

Now, we let first $\epsilon \to 0$ and then $\delta \to 0$.
To compute the first limit we rewrite our problem \eqref{Dir-ep-delsec} as follows: 
$$
\frac{\delta^2}{C} f(x) =\frac{1}{\delta^{N}} \int_{\R^N} J \left(\frac{x-y}{\delta}\right) 
(u^{\epsilon,\delta} (y) - u^{\epsilon,\delta} (x)) dy, \qquad x \in \Omega^\epsilon.
$$
Note that the involved kernel satisfies
$$
\frac{1}{\delta^{N}} \int_{\R^N} J \left(\frac{z}{\delta}\right) \, dz =1.
$$
This property was used in the proof of Theorem \ref{theoD} in Section \ref{sect-Dirichlet}.
Arguing as in Section \ref{sect-Dirichlet} with a fixed $\delta$ 
we obtain that there exists $u^{*,\delta} \in L^2(\Omega)$ such that 
$$\tilde u^{\epsilon, \delta} \rightharpoonup u^{*,\delta}  \textrm{ weakly in } L^2(\Omega).$$
Moreover, the limit $u^{*,\delta}$ satisfies the following nonlocal problem in $\Omega$
\begin{equation} \label{eq.jjj}
\begin{array}{l}
\displaystyle
\frac{\delta^2}{C}\mathcal{X}(x) \, f(x) = \mathcal{X}(x) \, \frac{1}{\delta^{N}} \int_{\R^N}  J \left( \frac{x-y}{\delta} \right) (u^{*,\delta} (y) - u^{*,\delta} (x) ) \, dy - (1-\mathcal{X}(x)) \, u^{*,\delta} (x) 
\end{array}
\end{equation}
with $u^{*,\delta} (x) \equiv 0$, for $x \in \R^N \setminus \Omega$.
Note that, since we are considering periodic holes, we have that $\mathcal{X}$ is a constant 
given by
\begin{equation} \label{ffff}
\mathcal{X}=
\left\{\begin{array}{ll}
1  \qquad & \mbox{ if }  r^\epsilon \ll C_0 \epsilon, \\[5pt]
\displaystyle \frac{|Q\setminus B|}{|Q|}  \qquad & \mbox{ if }  r^\epsilon = C_0 \epsilon.
\end{array} \right.
\end{equation}
Recall that $|Q\setminus B|$ is the measure of the complement of the ball $B$ in the cube $Q$. Hence, here
$\frac{|Q \setminus B|}{|Q|} $ is the proportion of the cube that is inside $\Omega^\epsilon$.
Also recall that in this case the critical size is 
$b^\epsilon := C_0 \epsilon$.

Now, for the limit as $\delta \to 0$ we consider two cases. First, when $r^\epsilon \ll b^\epsilon$,
we have $\mathcal{X} =1$ and \eqref{eq.jjj} reduces to
$$
f(x) =  
 \frac{C}{\delta^{N+2}} \int_{\R^N}  J \left( \frac{x-y}{\delta} \right) (u^{*,\delta} (y) - u^{*,\delta} (x) ) \, dy ,
$$
and from the results in \cite[Section 3.2.2]{ElLibro} we get that 
$$
u^{*,\delta} \to w, \qquad \mbox{ in }L^2 (\Omega),
$$ 
as $\delta \to 0$,
where $w$ is the solution to
$\Delta w = f $ with $w=0$ on $\partial \Omega$.

In the case $r^\epsilon = b^\epsilon$ we have 
$$
\frac{\delta^2}{C}  f(x) =  \frac{1}{\delta^{N}} \int_{\R^N}  J \left( \frac{x-y}{\delta} \right) (u^{*,\delta} (y) - u^{*,\delta} (x) ) \, dy - \frac{(1-\mathcal{X})}{\mathcal{X} } \, u^{*,\delta} (x) .
$$
Multiplying by $u^{*,\delta}$ and integrating we get
$$
\begin{array}{l}
\displaystyle
\int_{\R^N}  \frac{\delta^2}{C}  f(x) u^{*,\delta} (x)  \, dx= -\frac12 \frac{1}{\delta^{N}} \int_{\R^N} \int_{\R^N}  J \left( \frac{x-y}{\delta} \right) (u^{*,\delta} (y) - u^{*,\delta} (x) )^2 \, dy \, dx 
\\[10pt]
\qquad \qquad \qquad \qquad \qquad \qquad \qquad \displaystyle
- \frac{(1-\mathcal{X})}{\mathcal{X} }  \int_{\R^N}  (u^{*,\delta} (x) )^2 \, dx.
\end{array}
$$
It follows that
$$
 \frac{(1-\mathcal{X})}{\mathcal{X} }   \int_{\R^N}  (u^{*,\delta} (x) )^2 \, dx
\leq \left| \int_{\R^N}  \frac{\delta^2}{C}  f(x) u^{*,\delta} (x)  \, dx \right| \leq C(f) \delta^2
\left(  \int_{\R^N}  (u^{*,\delta} (x) )^2 \, dx\right)^{1/2},
$$
and we conclude that
$$
u^{*,\delta} \to 0, \qquad \mbox{ in } L^2 (\Omega),
$$
as $\delta \to 0$.

Hence, we have obtained that  
$$
\| u^{*,\delta} - w \|_{L^2(\Omega)} \to 0, \quad \textrm{ as } \delta \to 0,
$$
where $w$ is given by
\begin{equation}
\label{chichi}
w = \left\{
\begin{array}{ll}
\mbox{the solution to } \Delta w = f \mbox{ with } w=0 \mbox{ on }\partial \Omega
, \qquad & \, \mbox{if } r^\epsilon \ll b^\epsilon, \\[6pt] 
0, \qquad & \, \mbox{if } r^\epsilon = b^\epsilon.
\end{array}
\right.
\end{equation}

We have obtained the following theorem.
\begin{theorem} \label{teo.rescalesDir.1.67}
If $u^{\delta,\epsilon} $, $w$ are given by \eqref{Dir-ep-delsec} and \eqref{chichi} respectively,it holds that 
$$
\lim_{\delta \to 0} \left( \lim_{\epsilon \to 0} \tilde{u}^{\delta,\epsilon} \right) = w \quad \textrm{ weakly in } L^2(\Omega),
$$
\end{theorem}

Now, let us see when $v$ and $w$ coincide.
We have to distinguish several cases according to the size of the holes. Notice $b^\epsilon \geq a^\epsilon$ whenever $\epsilon$ is small enough.

{\bf Case 1. $r^\epsilon = b^\epsilon$.} In this case we have that $w =0$ and $v=0$. Therefore 
$w= v$ in this case. 

{\bf Case 2. $a^\epsilon \ll r^\epsilon \ll b^\epsilon$.}
In this case, $w$ is the solution to
$\Delta w = f $ and $v=0$. Therefore $w\neq v$ in this case. 

{\bf Case 3. $r^\epsilon = a^\epsilon$.}
In this case, $w$ is the solution to
$\Delta w = f $ and $v$ the solution to $\Delta v - \mu \, v = f $. Therefore, also $w\neq v$ in this case. 

{\bf Case 4. $r^\epsilon \ll a^\epsilon$.}
In this case, $w$ and $v$ coincide and are given by the unique solution to
$\Delta w = f $.

\subsection{The Neumann case}

Now we consider
\begin{equation} \label{Neu-ep-delsec}
f(x) =\frac{C}{\delta^{N+2}} \int_{\R^N \setminus A^\epsilon} J \left(\frac{x-y}{\delta}\right) 
(u^{\epsilon,\delta} (y) - u^{\epsilon,\delta} (x)) dy, \qquad x \in \Omega^\epsilon,
\end{equation}
with 
\begin{equation} \label{Cond.Dir.sec.88}
u^{\epsilon,\delta} (x) \equiv 0, \qquad x \in \R^N \setminus \Omega,
\end{equation}
and $C$ is the same normalizing constant that we used for the Dirichlet case and is given by 
\eqref{constC.Dir}.

Existence and uniqueness of the solutions $u^{\epsilon,\delta}$ of \eqref{Neu-ep-delsec} are guaranteed in Section \ref{sect-Neumann} for any $f \in L^2(\Omega)$, $\epsilon$ and $\delta>0$, provided  there is a positive constant $c$ independent of $\epsilon$ such that
$$
\lambda_1^\epsilon \geq c >0.
$$
We have seen that this holds in our case, that is, for periodically perforated domains.

Let us proceed with the analysis of the behavior of $u^{\epsilon,\delta}$ as $\epsilon$ and $\delta$ go to zero.  
From \eqref{NPRN.jj} using results from \cite{CERW} we have
$$
\| u^{\delta,\epsilon} - v^\epsilon \|_{L^2(\Omega^\epsilon)} \to 0, \quad \textrm{ as } \delta \to 0.
$$
being $v^\epsilon$ the solution to the local problem
\begin{equation} \label{LNP-mm}
\left\{
\begin{array}{ll}
\displaystyle \Delta v^\epsilon(x)  = f(x), \qquad & x \in \Omega^\epsilon \\[5pt]
v^\epsilon (x) = 0, 
\qquad & x \in \partial \Omega, 
\\[5pt]
\displaystyle
\frac{\partial  v^\epsilon(x)}{\partial \eta} = 0, \qquad & x \in \partial \Omega^\epsilon \cap \Omega.
\end{array}
\right. 
\end{equation}

Now, using results from \cite[Theorem 2.16]{CioS} (see also \cite{CioS2}) we get that  the limit as $\epsilon \to 0$ 
in this local problem is given by 
\begin{equation}
\label{chi-Neu}
v = \left\{
\begin{array}{ll}
\mbox{the solution to } \Delta v = f , \qquad & \mbox{if } r^\epsilon \ll b^\epsilon, \\[5pt] 
\displaystyle \mbox{the solution to } \sum_{i,j=1}^N  q_{ij} \frac{\partial^2 v}{\partial x_i \partial x_i} = \frac{|Q \setminus B|}{|Q|}f , \qquad & \mbox{if } r^\epsilon = b^\epsilon,
\end{array}
\right.
\end{equation}
with Dirichlet boundary conditions, $v=0$ on $\partial \Omega$.
Here $q_{ij}$ are called the \emph{homogenized coefficients} and are given by
$$
q_{ij} = \frac{1}{|Q|} \left[ \int_{Q \setminus B} \delta_{ij} \, dy - \int_{Q \setminus B} \frac{\partial X^i}{\partial y_j}(y) \, dy \right],
$$
where $\delta_{ij}$ is the Kronecker delta and $X^i$ (for $i=1,...,N$) are the solutions of the system
$$
\left\{ 
\begin{array}{ll}
\Delta X^i = 0 & \textrm{ in } Q \setminus B, \\
\partial_\eta X^i = \eta_i & \textrm{ on } \partial B, \\
X^i & Q-\textrm{periodic}
\end{array}
\right.
$$
with $\int_{Q \setminus B} X^i(y) \, dy = 0$.
Here the critical size of the holes is 
$
b^\epsilon := C_0 \epsilon$,

Now let $P_\epsilon$ be an extension operator, that is, $P_\epsilon \in \mathcal{L}(L^2(\Omega^\epsilon); L^2(\Omega)) \cap \mathcal{L}(V_\epsilon; H^1_0(\Omega))$ 
where $V_\epsilon = \{ v \in H^1(\Omega) \, : \, u = 0 \textrm{ on } \partial_{ext} \Omega^\epsilon \}$
(the existence of such extension operator was proved in \cite{CM}). 
Then, from \cite[Theorem 2.16]{CioS} we have 
$$
\| P_\epsilon v^\epsilon - v \|_{L^2(\Omega)} \to 0, \quad \textrm{ as } \epsilon \to 0.
$$

Thus, we conclude that

\begin{theorem} \label{teo.rescalesDir.1.77}
If $u^{\delta,\epsilon} $, $v$ are given by \eqref{Neu-ep-delsec} and \eqref{chi-Neu} respectively, it holds that 
$$
\lim_{\epsilon \to 0} P_\epsilon \left( \lim_{\delta \to 0} u^{\delta,\epsilon} \right) = v \quad \textrm{ in } L^2(\Omega).
$$
\end{theorem}

Now, we reverse the order in which we take limits and let first $\epsilon \to 0$ and then $\delta \to 0$.
First, as we did in the Dirichlet case we write our equation as 
$$
\frac{\delta^2}{C}f(x) =\frac{1}{\delta^{N}} \int_{\R^N \setminus A^\epsilon} J \left(\frac{x-y}{\delta}\right) 
(u^{\epsilon,\delta} (y) - u^{\epsilon,\delta} (x)) dy, \qquad x \in \Omega^\epsilon,
$$
since again in this case we want to use that the involved kernel satisfies
$$
\frac{1}{\delta^{N}} \int_{\R^N} J \left(\frac{z}{\delta}\right) \, dz =1.
$$

Arguing as in Section \ref{sect-Neumann} with a fixed $\delta$ 
we obtain that there exists a limit $u^{*,\delta} \in L^2(\Omega)$ such that 
$$\tilde u^{\epsilon, \delta} \rightharpoonup u^{*,\delta}  \qquad \textrm{ weakly in } L^2(\Omega).$$
Moreover, the limit $u^{*,\delta}$ satisfies the following nonlocal problem in $\Omega$
\begin{equation} \label{eq.mmm}
\left\{
\begin{array}{l}
\displaystyle
\frac{\delta^2}{C} \mathcal{X}(x) \, f(x) = \mathcal{X}(x) \, \frac{1}{\delta^{N}} \int_{\R^N}  J \left( \frac{x-y}{\delta} \right) (u^*(y) - u^*(x) ) \, dy -\Lambda(x) \, u^*(x) 
\\[10pt]
\displaystyle 
u^*(x) \equiv 0, \qquad x \in \R^N \setminus \Omega,
\end{array}
\right.
\end{equation}
where $\Lambda \in L^\infty(\Omega)$ is given by 
$$
\Lambda(x) =  \frac{1}{\delta^{N}} \int_{\R^N}  J \left( \frac{x-y}{\delta} \right) \, ( 1- \chi_\Omega(y) + \mathcal{X}(y) ) \, dy - \mathcal{X}(x), \quad x \in \Omega.
$$
Here $\chi_\Omega$ is the characteristic function of the open set $\Omega$ and $\mathcal{X}$ is the constant given by 
\eqref{ffff}, that is,
$$
\mathcal{X}=
\left\{\begin{array}{ll}
\displaystyle
\frac{|Q \setminus B|}{|Q|}  \qquad & \mbox{ if }  r^\epsilon = C_0 \epsilon, \\[10pt]
1  \qquad & \mbox{ if }  r^\epsilon \ll C_0 \epsilon.
\end{array} \right.
$$
Note that also in this case the critical size is 
$
b^\epsilon := C_0 \epsilon$.

Now, for the limit as $\delta \to 0$ we consider two cases. First, when $r^\epsilon \ll b^\epsilon$,
we have $\mathcal{X} =1$ and from the results in \cite[Section 3.2.2]{ElLibro} we get that 
$$
u^{*,\delta} \to w, \qquad \mbox{ in }L^2 (\Omega),
$$ 
as $\delta \to 0$,
where $w$ is the solution to
$\Delta w = f $ with $w=0$ on $\partial \Omega$.

In the case $r^\epsilon = b^\epsilon$ we have 
$$
\frac{\delta^2}{C}  f(x) =  \frac{1}{\delta^{N}} \int_{\R^N}  J \left( \frac{x-y}{\delta} \right) (u^{*,\delta} (y) - u^{*,\delta} (x) ) \, dy - \Lambda \, u^{*,\delta} (x) .
$$
Multiplying by $u^{*,\delta}$ and integrating we get, arguing as in the Dirichlet case and using Lemma \ref{lema.posi.neu}, that
$$
u^{*,\delta} \to 0, \qquad \mbox{ in } L^2 (\Omega),
$$
as $\delta \to 0$.

Hence, we have obtained that $u^{*,\delta}$ converges as $\delta \to 0$ to $w$ that is given by
\begin{equation}
\label{chichi-Neu}
w = \left\{
\begin{array}{ll}
\mbox{the solution to } \Delta w = f , \mbox{ with $w=0$ on $\partial \Omega$},\qquad & \mbox{if } r^\epsilon \ll b^\epsilon, \\[6pt] 
 0 , \qquad & \mbox{if } r^\epsilon = b^\epsilon.
\end{array}
\right.
\end{equation}

Therefore, we have the following theorem.
\begin{theorem} \label{teo.rescalesNeu.1}
If $u^{\delta,\epsilon} $, $w$ are given by \eqref{Neu-ep-delsec} and \eqref{chichi-Neu} respectively,it holds that 
$$
\lim_{\delta \to 0} \left( \lim_{\epsilon \to 0} \tilde{u}^{\delta,\epsilon} \right) = w \quad \textrm{ weakly in } L^2(\Omega).
$$
\end{theorem}

Now, let us see when $v$ and $w$ coincide.
We have to distinguish only two cases according to the size of the holes:

{\bf Case 1. $r^\epsilon = b^\epsilon$.} In this case we have that $w=0$ 
and $v$ is the solution to 
$$\sum_{i,j=1}^N  q_{ij} \frac{\partial^2 v}{\partial x_i \partial x_i} = \frac{|Q \setminus B|}{|Q|}f .$$
Note that $w\neq v$ in this case. 

{\bf Case 2. $r^\epsilon \ll b^\epsilon$.}
In this case, $w$ and $v$ coincide and are given by the unique solution to
$\Delta w = f $.

{\bf Acknowledgements.} 
The first author is partially supported by CNPq 302960/2014-7 and 471210/2013-7, FAPESP 2015/17702-3 (Brazil) and the second author by MINCYT grant MTM2016-68210 (Spain).


\begin{thebibliography}{BH}


\bibitem{GA} G. Allaire, {\it Homogenization of the Stokes flow in a connected porous medium}. Asymp. Anal. 2 (1989) 203-222.

\bibitem{ElLibro} F. Andreu-Vaillo, J. M. Maz\'{o}n, J. D. Rossi and J. Toledo,       Nonlocal Diffusion
Problems.   Mathematical Surveys and Monographs, vol. 165. AMS, 2010.

\bibitem{AB} J. M. Arrieta and S. M. Bruschi, {\it Very rapidly varying boundaries in equations with nonlinear boundary conditions. The case of non uniform Lispschitz deformation}. Discrete and Continuous Dynamical Systems B 14 (2) (2010) 327--351.

\bibitem{BPP} P. S. Barbosa, A. L. Pereira and M. C. Pereira, {\it Continuity of attractors for a family of $\mathcal{C}^1$
 perturbations of the square}. Annali di Matematica Pura ed Applicata 196 (2017) 1365--1398. 

\bibitem{Chasseigne} G. Barles, E. Chasseigne and C. Imbert, {\it On the Dirichlet problem for second-order elliptic integro-differential equations}. Indiana Univ. Math. J. 57 (2008), no. 1, 213--246.

\bibitem{ADMET} A. Brillard, D. G\'omez, M. Lobo, E. P\'erez and T. A. Shaposhnikovad,  {\it Boundary homogenization in perforated domains for adsorption problems with an advection term}. Appicable Analysis 95 (7) (2016) 1517--1533. 


\bibitem{caffa2} L. A. Caffarelli and A. Mellet, {\it Random homogenization of fractional obstacle problems}.
Netw. Heterog. Media 3 (2008), no. 3, 523--554.

\bibitem{CJM} C. Calvo-Jurado, J. Casado-D\'iaz and M. Luna-Laynez, {\it Homogenization of nonlinear Dirichlet problems in random perforated domains}. Nonlinear Analysis 133 (2016) 250--274.

\bibitem{CK} G. Cardone and A. Khrabustovskyi, {\it Neumann spectral problem in a domain with very corrugated boundary}. J. of Diff. Equations 259 (6) (2015) 2333--2367.

\bibitem{Ca} P. Cazeaux and C. Grandmont, {\it Homogenization of a multiscale viscoelastic model with nonlocal damping, application to the human lungs}. Math. Models Methods Appl. Sci. 25 (2015), no. 6, 1125--1177.

\bibitem{DAPGR} D. Cioranescu, A. Damlamian, P. Donato, G. Griso and R. Zaki, {\it The periodic unfolding method in domains with holes}. SIAM J. Math. Anal. 44 (2) (2012) 718--760.

\bibitem{CD} D. Cioranescu and P. Donato, An Introduction to Homogenization. Oxford lecture series in mathematics and its applications, vol.17, Oxford University Press, 1999.

\bibitem{CM} D. Cioranescu and F. Murat, {\it A strange term coming from nowhere}, Progress in Nonl. Diff. Eq. and Their Appl. 31 (1997) 45--93.


\bibitem{CioS} D. Cioranescu and J. Saint Jean Paulin, Homogenization of reticulated structures. Applied Mathematical Sciences, 136. Springer-Verlag, New York, 1999. 

\bibitem{CioS2} D. Cioranescu and J. Saint Jean Paulin, {\it Homogenization in open sets with holes}. J. Math. Anal. Appl. 71 (1979), no. 2, 590--607.

\bibitem{chfrt} E. Chasseigne, P. Felmer, J. D. Rossi and E. Topp, {\it Fractional decay bounds for nonlocal zero order heat equations}. Bull. Lond. Math. Soc. 46 (2014), no. 5, 943--952.

\bibitem{IP} I. Chourabi and P. Donato, {\it Homogenization and correctors of a class of elliptic problems in perforated domains}. Asymptotic Analysis 92 (2015) 1--43.

\bibitem{CER1} C. Cortazar, M. Elgueta and J. D. Rossi.
{\it Nonlocal diffusion problems that approximate the heat
equation with Dirichlet boundary conditions.} Israel J.
Math.~170(1), (2009), 53--60.

\bibitem{CERW} C. Cortazar, M. Elgueta, J. D. Rossi and N. Wolanski. {\it
How to approximate the heat equation with Neumann boundary
conditions by nonlocal diffusion problems.} Archive for Rational
Mechanics and Analysis.~187(1), (2008), 137--156.

\bibitem{CH} R. Courant and D. Hilbert, Methods of mathematical physics, Volume I. Interscience, New York, 1953.

\bibitem{Du} Q. Du, M. Gunzburger, R. B. Lehoucq and K. Zhou, {\it A nonlocal vector calculus, nonlocal volume-constrained problems, and nonlocal balance laws}. Math. Models Methods Appl. Sci. 23 (2013), no. 3, 493--540.

\bibitem{Ignat} L. I. Ignat, D. Pinasco, J. D. Rossi and A. San Antolin. {\it Decay estimates for nonlinear nonlocal diffusion problems in the whole space}. J. Anal. Math. 122 (2014), 375--401.

\bibitem{Ignat2} L. I. Ignat, T. Ignat and D. Stancu-Dumitru, {\it A compactness tool for the analysis of nonlocal evolution equations}. SIAM J. Math. Anal. 47 (2015), no. 2, 1330--1354.


\bibitem{jorge} J. Garc\'{\i}a Meli{\'a}n and J. D. Rossi. {\it On the
principal eigenvalue of some nonlocal diffusion problems.} J.
Differential Equations.~246(1), (2009), 21--38.

\bibitem{Per} R. B. Lehoucq and S. A. Silling,  {\it Force flux and the peridynamic stress tensor}. J. Mech. Phys. Solids 56 (2008), no. 4, 1566--1577.

\bibitem{MSZ} C. Mocenni, E. Sparacino and J. P. Zubelli, {\it Effective rough boundary parametrization for reaction-diffusion systems}. Applicable Analysis and Discrete Mathematics 8 (2014) 33--59.

\bibitem{NPS}  A. K. Nandakumaran, R. Prakash and B. C. Sardar, {\it Periodic Controls in an Oscillating Domain: Controls via Unfolding and Homogenization}. SIAM Journal on Control and Optimization 53 (5) (2015) 3245--3269.

\bibitem{Necas} J. Necas, Les m\'ethodes directes en th\'eorie des \'equations elliptiques. Masson, Paris, 1967.


\bibitem{N} G. Nguetseng, {\it A General Convergence Result for a Functional Related to the Theory of Homogenization}. SIAM J. Math. Anal. 20 (1989) 608--623. 

\bibitem{N2} G. Nguetseng, {\it Homogenization in perforated domains beyond the periodic setting}. J. Math. Anal. Appl. 289 (2004) 608--628.

\bibitem{JDE} M. C. Pereira and J. D. Rossi, {\it Nonlocal problems in thin domains}. J. Differential Equations 263 (2017) 1725--1754.

\bibitem{PtAn} M. C. Pereira and J. D. Rossi, {\it An Obstacle Problem for Nonlocal Equations in Perforated Domains}. To appear in Potential Anal DOI 10.1007/s11118-017-9639-5.

\bibitem{MTM} M. E. P\'erez, T. A. Shaposhnikova and M. N. Zubova, {\it A homogenization problem in a domain perforated by tiny isoperimetric holes with nonlinear Robin type boundary conditions}. Dokl. Math. 90 (2014) 489--494.


\bibitem{RT} J. Rauch and M. Taylor, {\it Potential and scattering theory on wildly perturbed domains}. J. Funct. Anal. 18 (1975), 27--59.

\bibitem{RS} A. Rodr\'iguez-Bernal and S. Sastre-G\'omez, {\it Linear nonlocal diffusion problems in metric measure spaces}. The Royal Society of Edinburgh Proceedings A 146 (2016) 833--863.

\bibitem{ESP} E. Sanchez-Palencia, Non Homogeneous Media and Vibration Theory. Lecture Notes in Physics, 127. Springer, Berlin, 1980.


\bibitem{sw2}  R. W. Schwab, {\it Periodic homogenization for nonlinear integro-differential equations}. SIAM J. Math. Anal. 42 (2010), no. 6, 2652--2680.

\bibitem{Wa} M. Waurick, {\it Homogenization in fractional elasticity}. SIAM J. Math. Anal. 46 (2014), no. 2, 1551--1576.


\end{thebibliography}
\end{document}